\documentclass[12pt]{amsart}
\usepackage[all]{xy}
\usepackage{amscd}
\usepackage{amssymb}
\usepackage{amsthm}
\theoremstyle{plain}
\newtheorem{thm}{Theorem}[section]
\newtheorem{prop}[thm]{Proposition}
\newtheorem{cor}[thm]{Corollary}
\newtheorem{lem}[thm]{Lemma}
\newtheorem{conj}[thm]{Conjecture}

\newtheorem{example}[thm]{Example}

\newtheorem{sblem}[thm]{Lemma}
\newtheorem{claim}[thm]{Claim}
\newtheorem{rem}[thm]{Remark}
\newtheorem{defn}[thm]{Definition}
\theoremstyle{definition}

\newenvironment{pf}{\proof[\proofname]}{\endproof}

\newcommand{\Cal}{\mathcal}

\define\Ker{\mathrm{Ker}\,}%

\newcommand{\F}{{\mathbb{F}}}

\newcommand{\Q}{{\mathbb{Q}}}
\newcommand{\Z}{{\mathbb{Z}}}

\newcommand{\upc}[1]{\overset {\lower 0.3ex \hbox{${\;}_{\circ}$}}{#1}}

\newcommand{\La}{\operatorname{\Lambda}}

\newcommand{\sel}{\mathrm{Sel}}
\newcommand{\cS}{\operatorname{\Cal S}}

\newcommand{\Ep}{E_{p^\infty}}

\newcommand{\Hfl}{H^1_{\mathrm{fl}}}

\newcommand{\lra}{\longrightarrow}
\newcommand{\uset}{\underset}
\newcommand{\cal}{\mathcal}
\newcommand{\mrm}{\mathrm}
\newcommand{\mbb}{\mathbb}
\newcommand{\mc}{\mathcal}
\newcommand{\oline}{\overline}

\newcommand{\finf}{F_{\infty}}
\newcommand{\Fpinf}{{F'}_{\infty}^{,(p)}}
\newcommand{\fpinf}{{F}_{\infty}^{(p)}}

\newcommand{\what}{\widehat}
\newcommand{\bd}[1]{\mbox{\boldmath$#1$}}
\title{On the Selmer groups of abelian varieties over function fields of 
characteristic $p>0$}
\author{Tadashi Ochiai and Fabien Trihan}
\thanks{The first author is supported by JSPS. 
The second author has been supported by the FNRS of Belgium.}
\begin{document}
\maketitle
\begin{abstract}
In this paper, we study a ($p$-adic) geometric analogue for 
abelian varieties over a function field of characteristic $p$ 
of the cyclotomic Iwasawa theory and the non-commutative Iwasawa theory 
for abelian varieties over a number field initiated by Mazur and Coates 
respectively. We will prove some analogue of the principal results 
obtained in the case over a number field and we study new phenomena 
which did not happen in the case of number field case. 
We propose also a conjecture (Conjecture $\ref{conj:ourconj}$) 
which might be considered 
as a counterpart of the principal conjecture in the case over a number 
field. \par 
This is a preprint which is distributed since 2005 which 
is still in the process of submision. 
Following a recent modification of some technical mistakes in the 
previous version of the paper as well as an amelioration 
of the presentation of the paper, we decide wider distribution 
via the archive. 
\end{abstract}
\tableofcontents
\begin{section}{Introduction} 
Iwasawa theory was first initiated by Iwasawa who 
studied ideal class groups in the cyclotomic $\Z_p$-extension 
$K^{\mrm{cyc},(p)}$ of a number field $K$ 
for a fixed prime $p$, which he regarded 
as an analogue of the extension of the coefficient field in the theory 
of algebraic curves over a finite field. 
He formulated the Iwasawa main conjecture, which predicts 
the equality between 
the algebraic $p$-adic $L$-function and the analytic $p$-adic $L$-function 
and was solved later by Mazur-Wiles \cite{MW} and Wiles \cite{wiles}.
\par 
It was Mazur \cite{M} who first pointed out a plan of extending the philosophy 
of the Iwasawa theory to higher dimensional motives like elliptic curves 
over number fields in early 70's. 
Since then, the Iwasawa theory was generalized to ordinary motives 
and it seems to be extended to Selmer groups of Galois deformation spaces 
constructed by the theory of Mazur, Hida, Coleman, etc. \par 
However, there is another direction of generalization called 
``non-commutative Iwasawa theory of elliptic curves" 
initiated by Coates.
\par 
Let us recall the classical (cyclotomic) Iwasawa theory for 
elliptic curves and its non-abelian generalization by Coates. 
Let $E$ be an elliptic curve over $K$. 
For any algebraic extension $L/K$, the Selmer group $\sel(E/L)$ 
is defined to be 
$$
\sel (E/L)=\Ker \left[ H^1_{\mrm{Gal}}(L,\Ep)\to\prod_v H^1_{\mrm{Gal}}
(L_v,E(\oline{L}_v))\right],
$$
where $v$ runs over all primes of $L$, $\Ep$ is the group of all 
$p$-power division points on $E$, and $H^1_{\mrm{Gal}}$ 
means the Galois cohomology. 
We denote by $X(E/L)$ the Pontrjagin dual 
$\mrm{Hom}_{\mrm{cont}}(\sel(E/L),\Q_p/\Z_p)$ 
of $\sel(E/L)$. For the cyclotomic Iwasawa theory for 
elliptic curves, we consider $X(E/K^{\mrm{cyc},(p)})$ 
which is a compact $\Z_p [[\Gamma^{(p)} ]]$-module where $\Gamma^{(p)} 
=\mrm{Gal}(K^{\mrm{cyc},(p)}/K) \cong \Z_p$. 
It is not difficult to show that $X(E/K^{\mrm{cyc},(p)})$ 
is a finitely generated $\Z_p [[\Gamma^{(p)}]]$-module 
(cf. \cite{M}, \cite{manin}). 
\begin{conj}\label{conj:mazur1}
Suppose that $E$ has good ordinary reduction at every primes 
of $K$ over $p$. Then 
$X(E/K^{\mrm{cyc},(p)})$ is a torsion $\Z_p [[\Gamma^{(p)}]]$-module.
\end{conj}
Under this conjecture, the algebraic $p$-adic $L$-function 
$f^{\mrm{alg}}_E \in \Z_p [[\Gamma^{(p)}]]$ 
is defined to be the characteristic polynomial of 
$X(E/K^{\mrm{cyc},(p)})$ via the structure theorem 
of finitely generated $\Z_p [[\Gamma^{(p)}]]$-modules, 
which is roughly the product 
annihilators of the torsion $\Z_p [[\Gamma^{(p)}]]$-module 
$X(E/K^{\mrm{cyc},(p)})$ modulo its maximal 
pseudo-null $\Z_p [[\Gamma^{(p)}]]$-submodule. 
Here is the Iwasawa Main Conjecture in this case$:$ 
\begin{conj}[Mazur]\label{conj:mazur}
Suppose that $E$ has good ordinary reduction at every primes 
of $K$ over $p$. Assuming the conjecture $\ref{conj:mazur1}$, 
The algebraic $p$-adic $L$-function $f^{\mrm{alg}}_E$ 
is equal to the analytic $p$-adic $L$-function 
$f^{\mrm{anal}}_E \in \Z_p [[\Gamma^{(p)} ]]$ constructed by 
Mazur and Swinnerton-Dyer modulo multiplication by a unit of 
$\Z_p[[\Gamma^{(p)} ]]$. 
\end{conj}
\begin{rem}\label{rem:supsing}
\begin{enumerate}
\item 
When there is a prime lying over $p$ at which the elliptic curve 
$E$ has good supersingular reduction, it is known that 
$X(E/K^{\mrm{cyc},(p)})$ is not a torsion $\Z_p [[\Gamma^{(p)}]]$-module.
\item 
By Euler system method Rubin \cite{rubin} proves 
Conjecture $\ref{conj:mazur1}$ and Conjecture 
$\ref{conj:mazur}$ for elliptic curves over $\Q$ with complex multiplication 
and Kato \cite{kato} 
proves Conjecture $\ref{conj:mazur1}$ and the divisibility 
$f^{\mrm{alg}}_E 
\vert f^{\mrm{anal}}_E$ of Conjecture 
$\ref{conj:mazur}$ for large class of elliptic curves over $\Q$ without 
complex multiplication. 
\end{enumerate}
\end{rem}
Recently, Skinner-Urban announces to have shown 
Conjecture $\ref{conj:mazur}$ for large class of elliptic curves 
over $\Q$ without complex multiplication assuming the 
conjecture on the existence of Galois representations for the 
modular forms on $U(2,2)$.
\par 
From several reasons, Coates proposed to pursue the analogue of Conjecture 
$\ref{conj:mazur}$ for non-commutative $p$-adic Lie extension $L$ of $K$.
We recall two examples of such $p$-adic Lie extension $L$. 
\begin{enumerate}
\item 
Suppose that $K$ contains $p$-th root of unity $\zeta_p$. 
We take $L$ to be a Kummer extension of $K^{\mrm{cyc},(p)}=
\cup_{n\geq 1}K(\mu_{p^n})$ defined by 
$L=\cup_{m\geq 1}K^{\mrm{cyc},(p)}(\sqrt[p^m]{a})$ for a fixed element 
$a \in K$. The Galois extension of $L/K$ has a non-split extension:
$$
0 \lra \Z_p \lra \mrm{Gal}(L/K) \lra \Z_p \lra 0. 
$$
\item 
Let $E$ be an elliptic curve over $K$ with no complex multiplication 
over a number field $\oline{K}$. 
Let $L:=K (\Ep)$ be the field obtained by adjoining $\Ep$ to $K$. 
By Weil pairing, $L$ contains $K^{\mrm{cyc},(p)}$. 
By a result of Serre, $\mrm{Gal}(L/K)$ is 
identified with an open subgroup of $GL_2(\Z_p)\simeq \mrm{Aut}(\Ep)$.
\end{enumerate}
From now on through the paper, we will denote by 
$\La (G)$ the algebra $\Z_p [[G]]$ for any profinite group $G$.
For $G= \mrm{Gal}(L/K)$, $X(E/L)$ has a natural structure of a left 
$\La (G)$-module. 
The algebra $\La (G)$ is no more commutative, but have natural notion 
of ``torsion $\La (G)$-module". 
A non-commutative generalization of Conjecture $\ref{conj:mazur1}$
is known as follows: 
\begin{conj}\label{conj:mazur1noncomm}
Suppose that $E$ has good ordinary reduction at every primes 
of $K$ over $p$. Assume that $L$ contains $K^{\mrm{cyc},(p)}$. 
Then $X(E/L)$ is a torsion $\La (G)$-module.
\end{conj}
To have an algebraic $p$-adic $L$-function is a necessary and important 
step for the study of Iwasawa theory. 
As in the classical cyclotomic case where $G\cong \Z_p$, 
we hope that the algebraic $p$-adic $L$-function is associated 
to $X(E/L)$ once Conjecture $\ref{conj:mazur1noncomm}$ is true. 
However, it is not impossible to associate a characteristic polynomial 
to a torsion $\La (G)$-module when $G$ is non commutative. 
After rather negative results showing 
difficulties of non-commutative Iwasawa theory, \cite{CFTKWT} 
gave a convincing formalism of the Selmer group and the algebraic 
$p$-adic $L$-function inspired by the habilitation thesis 
of Venjakob at Heidelberg 
and based on calculation of K-groups of non-commutative 
Iwasawa algebras by Kato. Assuming that $G$ has a normal subgroup $H$ 
such that $G/H \cong \Z_p$, 
The important idea is that the $p$-adic $L$-function 
lives in a $K$-group $K_1 (\La (G)_{S^{\ast}}) /K_1 (\La (G))$, 
where $\La (G)_{S^{\ast}}$ is the localization 
of $\La (G)$ by a certain multiplicative set $S^{\ast}$.
  We refer the reader to \cite{CFTKWT} for how to choose $S^{\ast}$ 
and the basic results related to $S^{\ast}$. 
We remark that $S^{\ast}$ consists of every non-zero elements 
in $\La (G)$ in the classical case $G=\Z_p$ and we have 
$K_1 (\La (G)_{S^{\ast}}) /K_1 (\La (G)) \cong \mrm{Frac}(\La (G) )^{\times} 
/\La (G) ^{\times}$ in this case. 
Since the characteristic power series is taken 
in a representative element of $\mrm{Frac}(\La (G) )^{\times} 
/\La (G) ^{\times}$, it seems to be reasonable that 
the $p$-adic $L$-function should be constructed in 
$K_1 (\La (G)_{S^{\ast}}) /K_1 (\La (G))$. 
Under the existence of a closed subgroup $H$ with $G/H\cong \Z_p$, 
\cite{CFTKWT} gives an exact sequence: 
$$
K_1 (\La (G)) \lra K_1 (\La (G)_{S^{\ast}}) \lra K_0 
(\mathfrak{M}_H) , 
$$ 
where $\mathfrak{M}_H$ is the category of finitely generated 
$\La (G)$-modules $X$ such that $X/X(p)$ is finitely generated as 
$\La (H)$-module, where $X(p)$ means the maximal $p$-power torsion 
submodule of $X$. 
According to \cite{CFTKWT}, $\La (G)$-modules which belong to 
$\mathfrak{M}_H$ satisfy all necessary properties 
which $p$-adic $L$-functions should have, related to 
the characteristic power series, the evaluation at characters of $G$ 
and the Euler-Poincar\'{e} characteristic. 
Hence, in order to have an algebraic $p$-adic $L$-function, 
we need the following conjecture stronger than 
Conjecture $\ref{conj:mazur1noncomm}$: 
\begin{conj}\cite{CFTKWT} 
Assume that $E$ has ordinary reduction at every primes 
of $K$ over $p$. 
Then, $X(E/L)/X(E/L)(p)$ is a finitely generated $\La(H)$-module.\\
\end{conj}
\par 
In this paper, we intend to discuss an analogue of the commutative 
and non-commutative Iwasawa theory in the case of abelian varieties 
over a function field $F$ in one variable over a finite field $\mbb{F}_q$.  
Let $A/F$ be an abelian variety. 
For any algebraic extension $K/F$, the Selmer group of $E/K$ 
in this context is by definition 
\begin{equation}\label{equation:seldef}
\sel (A/K):=\Ker \left[ \Hfl (K,A\{p\})\lra \prod_v \Hfl (K_v,A) \right],
\end{equation}
where $A\{p\}$ denotes the $p$-divisible group associated to $A$, 
$H^1_{\mrm{fl}}$ denotes the flat cohomology and $v$ runs all primes of $K$. 
\begin{enumerate}
\item 
Let $F_{\infty} :=F \overline{\mbb{F}}_q$, 
for an algebraic closure $\overline{\mbb{F}}_q$ of $\F_q$. 
The Galois group $F_{\infty}/F$ is isomorphic to $\widehat{\Z}$. 
We will define $F^{(p)}_{\infty} $ to be $F^{(p)}_{\infty} =F \mbb{F}^{(p)}_q$, where 
$\mbb{F}^{(p)}_q$ is the unique subfield of $\oline{\mbb{F}}_q$ 
such that $\mrm{Gal}(\mbb{F}^{(p)}_q /\mbb{F}_q)$ is isomorphic to $\Z_p$. 
\item 
Let $F^{\mrm{sep}}$ be a fixed separable closure of $F$. 
It is known that the maximal pro-$p$ quotient of 
$\mrm{Gal}(F^{\mrm{sep}}/F)$ is a free pro-$p$ group of infinite rank.
Hence any topologically finitely generated pro-$p$ group $G$ is 
realized as a Galois group $\mrm{Gal}(L/F)$ of certain Galois extension. 
\end{enumerate}
Now we will state below our result in this paper. 
We denote by $X(A/L)$ the Pontrjagin dual of $\sel (A/L)$, which 
has a natural structure of left $\La (G)$-module.  
We define $Y(A/L)$ to be $X(A/L)/X(A/L)(p)$, which is again 
a left $\La (G)$-module.\\ 
We propose the following conjecture: 
\begin{conj}\label{conj:ourconj}
Let $A$ be an abelian variety over $F$ with arbitrary reduction. 
Let $L$ be a Galois extension of $F$ containing $F^{(p)}_{\infty}$ 
whose Galois group $G=\mrm{Gal}(L/F)$ is a finitely generated 
torsion-free pro-$p$ group. We have the following conjecture. 
\begin{enumerate}
\item 
When $L$ is equal to $F^{(p)}_{\infty}$, $X(A/L)$ is a finitely generated 
torsion $\La (G)$-module. 
\item 
When $L$ is an extension containing $F^{(p)}_{\infty}$ such that 
$\mrm{Gal}(L/F)$ is a $p$-adic Lie extension, 
$Y(A/L)$ is a finitely generated $\La (H)$-module where 
$H=\mrm{Gal}(L/F^{(p)}_{\infty} )$. 
\end{enumerate}
\end{conj}
Our main results over $\fpinf$ are as follows: 
\begin{thm}\label{thm:cyc} 
Let $A$ be an abelian variety with arbitrary reduction. 
Then, $X(A/F^{(p)}_{\infty} )$ is a finitely generated torsion 
$\La (G )$-module for $G =\mrm{Gal}(F^{(p)}_{\infty} /F)$.\\
\end{thm}
Theorem $\ref{thm:cyc}$ will be proved in \S $3$. 
We insist that the independence of the reduction of $A$ for 
$X(A/F^{(p)}_{\infty} )$ to be torsion over $\La (G )$ 
is a different phenomenon when we compare our situation 
to that of the case of abelian varieties over the cyclotomic tower 
$K^{\mrm{cyc}, (p)}$ of a number field $K$ 
(cf. Conjecture $\ref{conj:mazur1}$). 
\par 
For a torsion $\La (G )$-module $M$ with $G \cong \Z_p$, 
we denote by $\mu (M)$ the length of $M_{(p)}$ over $\La (G)_{(p)}$, 
where $(\ \ )_{(p)}$ means the localization at the height-one 
prime ideal $(p)$. 
The invariant $\mu (M)$ is called the $\mu$-invariant of $M$ and $\mu (M)$ is 
an important numerical invariant of $M$. 
It is known that $\mu (X (A/K^{\mrm{cyc}, (p)}))$ is 
strongly dependent of the mod-$p$ representation of $A$. 
As the following result shows, the the $\mu$-invariant 
$\mu (X (A/\fpinf ))$ for an abelian variety $A$ over a function 
field $F$ seems to be strongly influenced by 
the reduction type of $A$. 
\begin{thm}\label{thm:cycmu}
\begin{enumerate}
\item 
Assume that there exists a finite separable extension $F'$ of $F$ 
which satisfies one of the following conditions: 
\vspace*{5pt} \ 
\\ 
{\bf (OF)} $A\times_F F'$ is isomorphic to an ordinary abelian variety 
defined over a finite field.  
\\ 
{\bf (SF)} $A\times_F F'$ is isomorphic to a supersingular abelian variety 
defined over a finite field and the proper smooth curve $C_{F'}$ which is 
the model of $F'$ has invertible Hasse-Witt matrix.
\vspace*{5pt}
\\ 
Then, we have $\mu (X(A/\fpinf ))=0$. 
\item 
Let $A$ be an isotrivial abelian variety 
having everywhere supersingular reduction.  
We have $\mu (X(A/\fpinf ))=0$ if and only if 
there exists a finite extension $F'/F$ such that 
the Hasse-Witt matrix for $F'$ is invertible and that 
$A\times_F F'$ is isomorphic to an abelian variety 
defined over a finite field. 
\end{enumerate}
\end{thm}
Theorem $\ref{thm:cycmu}$ will be proved in \S $3$. 
The following results give an evidence for Conjecture 
$\ref{conj:ourconj}$. 
\begin{thm}\label{thm:isotriv}
Let $L$ be a Galois extension of $F$ containing $F^{(p)}_{\infty} $ 
whose Galois group $G=\mrm{Gal}(L/F)$ is a $p$-adic Lie group with 
no non-trivial torsion elements. Let $A$ be an abelian variety 
over $F$. We denote by $H$ the Galois group $\mrm{Gal}(L/F^{(p)}_{\infty} )$. 
Let us assume the following conditions: 
\begin{enumerate}
\item 
$L/F^{(p)}_{\infty}$ is ramified at only finitely many primes of 
$F^{(p)}_{\infty}$, say $S$. 
\item 
$A$ has good reduction outside $S$. 
For each prime $v\in S$, $A$ has ordinary reduction $($not necessarily good 
ordinary reduction$)$. 
\item 
We have $\mu (X(A/\fpinf ))=0$. 
\end{enumerate}
Then, $X(A/L)$ is a finitely generated $\La(H)$-module.\\
\end{thm}
Theorem $\ref{thm:isotriv}$ will be proved in \S $4$. 
\begin{cor}
Assume that there exists a finite separable extension $F'$ 
satisfying the condition {\bf (OF)}.
Then, $X(A/L)$ is a finitely generated $\La(H)$-module.
\end{cor}
\par 
In this paper, we decided not to prove our result in full generality. 
However, in proving our results, the necessary hypothesis on the 
reduction of the abelian variety was rather general as compared to 
the cases of abelian varieties over number fields. 
Thus, Main Conjecture (Conjecture $\ref{conj:ourconj}$) was stated 
under more general condition than the one we imagined 
at the beginning of our work. In subsequent papers, 
we will generalize our results by 
removing some of hypothesis assumed in our theorems and we would 
like to make further research which is expected from 
the philosophy of the Iwasawa theory. \\ 
\par 
{\bf Acknowledgements.} 
The first author would like to thank Yoshitaka Hachimori 
for useful discussion on non-commutative generalization 
of Iwasawa theory. 
The second author would like to thank Kazuya Kato
for introducing him to the theory of Iwasawa and Takeshi Saito and the 
University of Tokyo for their hospitality. 
The authors are grateful to Kato for pointing out 
technical mistakes in the first draft of the paper. 
They are also grateful to anonymous refrees for 
pointing out historical mistakes and technical mistakes 
in \S $\ref{section:genord}$. 
\end{section}
\section{General abelian varieties over the base $F^{(p)}_{\infty} $}
\label{section:genord}
In this section, we will prove Theorem $\ref{thm:cyc}$. 
First, we will need the following lemma: 
\begin{lem}\label{lem:ssred}
Let $F'/F$ be a finite Galois extension. 
We denote by $A'$ the extension $A'=A \times_F F'$.
If $X(A'/\Fpinf )$ is a finitely generated torsion 
$\La (G')$-module for $G'=\mrm{Gal}(\Fpinf /F')$, 
$X(A/F^{(p)}_{\infty})$ is a finitely generated torsion 
$\La (G)$-module for $G=\mrm{Gal}(F^{(p)}_{\infty}/F)$. 
\end{lem}
\begin{proof}
Suppose that 
$X(A'/\Fpinf )$ is a finitely generated torsion 
$\La (G')$-module. 
Then we have $J= \mrm{Gal}(F'/F^{(p)}_{\infty} \cap F')\cong 
\mrm{Gal}(\Fpinf /F^{(p)}_{\infty})$. 
\begin{equation}
\xymatrix{
   &  {F'}_{\infty}^{,(p)}  \ar@{-}_{G'}[dl] \ar@{-}^{J}[d]\\
 F' \ar@{-}[d]_{J} & F^{(p)}_{\infty} \ar@{-}_{G'}[dl] \\ 
 F' \cap F^{(p)}_{\infty} \ar@{-}[d] &  \\
 F \ar@{-}_{G}[uur] & }
\end{equation}
Recall that $X(A /\fpinf )$ and $X(A' /\Fpinf )$ are quotients of 
the modules $H^1_{\mrm{fl}} (\fpinf ,A\{ p\})^{\vee}$ and 
$H^1_{\mrm{fl}} (\Fpinf ,A\{ p\})^{\vee}$ 
respectively by the definition 
given in $(\ref{equation:seldef})$. 
Hence the cokernel of the natural $\Lambda (G')$-linear map 
$X(A'/\Fpinf)_J \lra X(A/\fpinf )$ 
is a quotient of the Pontrjagin dual of a finite group 
$H^1 (J ,A\{ p\}(\Fpinf )) \cong \mrm{Ker}[
H^1_{\mrm{fl}} (\fpinf ,A\{ p\}) \lra 
H^1_{\mrm{fl}} (\Fpinf ,A\{ p\})^{J}]$. Hence the $\Lambda (G)$-module 
$X(A/\fpinf )$ is a torsion $\La (G')$-module via 
natural identification of $G'$ as an open subgroup of $G$. 
Hence it is also torsion over $\La (G)$. 
\end{proof}
By the semi-stable reduction theorem for abelian varieties, 
an abelian variety $A$ over $F$ has everywhere semi-stable reduction 
after the base field extension $A\times _F F'$ by a certain finite 
Galois extension $F'/F$. 
Hence, by Lemma $\ref{lem:ssred}$, we may (and we will) assume the following 
condition {(SS)} from now on through this section for the proof 
of Theorem $\ref{thm:cyc}$:  \vspace*{5pt}
\\ 
{\bf (SS)} $A/F$ has everywhere semi-stable reduction. 
\vspace*{5pt}
\\ 
Let $C_{\finf}$ be a proper smooth geometrically connected curve 
over $\oline{\mbb{F}}_q$ which is the model of the function field $\finf 
=F\oline{\mbb{F}}_q $. 
We denote by $U$ the dense open subset of $C_{\finf}$ such that 
$A \times _F \finf $ has good reduction on $U$ and we denote by 
$Z$ the complement $C_{\finf} \setminus U$. Let $\mathcal{A}$ be the 
N\'{e}ron model of $A \times _F \finf $ over $C_{\finf}$. 
We associate the Lie algebra 
$\mrm{Lie}(\mc{A})$ to $\mc{A}$, which is a sheaf of 
algebras on $C_{\finf}$. 
\par 
In \cite{KT},  syntomic cohomology 
for abelian varieties has been studied 
for application to the Birch and Swinnerton-Dyer conjecture of 
abelian varieties over function fields. 
In this section, we will reduce our main result to 
Theorem $\ref{theorem-section2}$ for general Dieudonn\'{e} 
crystals using tools studied in \cite{KT} and give a proof of 
Theorem $\ref{theorem-section2}$. 
\par 
We shall summarize necessary results and definition for later use 
in this section. From now on, 
we denote the Witt algebra $W(\oline{\mbb{F}}_q)$ by $W$ 
if there is no confusion. Let us denote by $C_{\finf }^{\sharp}$ 
the log scheme associated to a divisor $Z$ on $C_{\finf }$. 
We refer the reader to \cite{KT} for detailed explanation. 
Let $D$ be a Dieudonn\'{e} crystal on $C_{\finf }^{\sharp}/W$. 
If $Z$ is empty, then $D$ is a classical Dieudonn\'{e} 
crystal corresponding to a $p$-divisible group $G/C_{\finf}$. 
Recall the following theorem from \cite{KT}:
\begin{thm} Let $D$ be a Dieudonn\'{e} 
crystal on $C_{\finf }^{\sharp}$. Let $i$ be the canonical morphism of 
topos of \cite{bbm} from the topos of sheaves on 
$(C_{\finf})_{\mrm{\acute{e}t}}$ 
to the crystalline topos $(C_{\finf }^{\sharp}/W)_{\mrm{crys}}$. 
There exists an 
${\cal O}_{C_{\finf }}$-module $\mrm{Lie}(D)$ 
locally free of finite rank and a surjective map of sheaves 
$D\to i_*(\mrm{Lie}(D))$ in $(C_{\finf }^{\sharp}/W)_{\mrm{crys}}$.
\end{thm}
\begin{pf} The proof of \cite[\S 5.3]{KT}, remains 
still valid if we replace a base scheme over $\F_q$ 
by a scheme over $\oline{\F}_q$.
\end{pf}
Let us summarize necessary results and definitions related to the above 
theorem. 
\begin{enumerate}
\item We denote by $D^0$ the kernel of 
$D\lra i_{\ast}(\mrm{Lie}(D))$ in 
$(C_{\finf }^{\sharp}/W )_{\mrm{crys}}$ and 
we denote the canonical injection by ${\bf 1}:\ D^0\lra D$. 
By applying the canonical projection $u_{\ast}$ 
from the crystalline topos $(C_{\finf }^{\sharp}/W)_{\mrm{crys}}$ 
to the topos of sheaves on $(C_{\finf })_{\mrm{\acute{e}t}}$, 
we get a distinguished triangle$:$ 
$$
Ru_{\ast} D^0 \overset{\bf 1}{\lra}  Ru_{\ast} D\lra \mrm{Lie}(D).
$$
We can twist this triangle by the divisor $Z$ to get 
a triangle: 
\begin{equation}\label{equation:tri}
Ru_{\ast} D^0 (-Z) 
\overset{\bf 1}{\lra}  Ru_{\ast} D (-Z) \lra \mrm{Lie}(D)(-Z).
\end{equation}
where $D(-Z)$ is the twist of 
the Dieudonn\'{e} crystal $D$ defined in \cite[\S 5.11]{KT}. 
\item 
We denote the $i$-th cohomology associated to 
$Ru_{\ast} D^0$ $($resp. $Ru_{\ast} D)$ by the symbol 
$H^i_{\mrm{crys}}(C_{\finf}^{\sharp}/W,D^0(-Z))$ 
$($resp. $H^i_{\mrm{crys}}(C_{\finf}^{\sharp}/W,D(-Z)))$ and by $H^i_{\mrm{crys}}(C_{\finf}^{\sharp}/W,D^0(-Z)\otimes \Q_p/\Z_p)$ (resp. $H^i_{\mrm{crys}}(C_{\finf}^{\sharp}/W,D(-Z)\otimes \Q_p/\Z_p)$) the $i$-th cohomology associated to 
$Ru_{\ast} D^0\otimes^{\mbb{L}} 
\Q_p/\Z_p$ (resp. $Ru_{\ast} D\otimes^{\mbb{L}} 
\Q_p/\Z_p$) by abuse of notation. 
We have the following lemma: 
\begin{lem}\label{lem1-section2} \ \\
\begin{enumerate}
\item The canonical map 
$H^i_{\mrm{crys}}(C_{\finf}^{\sharp}/W,D^0 (-Z))\buildrel{\bf 1}\over\lra 
H^i_{\mrm{crys}}(C_{\finf}^{\sharp}/W,D(-Z))$ has a kernel and a cokernel
  killed by $p$.
\item $H^i_{\mrm{crys}}(C_{\finf}^{\sharp}/W,D^0(-Z))$ and
  $H^i_{\mrm{crys}}(C_{\finf}^{\sharp}/W,D(-Z))$ are finitely generated 
$W$-modules with the same rank.
\end{enumerate}
\end{lem}
\begin{pf}  
By taking the cohomology of the triangle 
$(\ref{equation:tri})$, we have the exact sequence: 
\begin{multline*}
H^{i-1}_{\mrm{fl}}(C_{\finf},\mrm{Lie}(D)(-Z)) \lra 
H^i_{\mrm{crys}}(C_{\finf}^{\sharp}/W,D^0(-Z)) \\ 
\buildrel{\bf 1}\over\lra 
H^i_{\mrm{crys}}(C_{\finf}^{\sharp}/W,D(-Z)) 
\lra H^{i}_{\mrm{fl}}(C_{\finf},\mrm{Lie}(D)(-Z)). 
\end{multline*} 
Since $H^{j}_{\mrm{fl}}(C_{\finf},\mrm{Lie}(D)(-Z))$ 
is a finite dimensional vector space over $\oline{\F}_q$ 
for every $j$, this completes the proof of the first assertion. 
The second assertion follows from the first one and the 
fact that crystalline cohomologies over a proper log smooth 
base scheme with finite locally free coefficients 
are of finite dimension by Tsuji \cite{T}. 
\end{pf}
\item 
As in \cite[\S 5.8]{KT}, we can construct a Frobenius operator 
$$\varphi:Ru_*D^0(-Z)\lra Ru_{\ast}D(-Z),$$ 
which induces a $\sigma$-linear homomorphism 
$$
\varphi_i  : 
H^i_{\mrm{crys}}(C_{\finf}^{\sharp} /W,D^0 (-Z)) 
\lra 
H^i_{\mrm{crys}}(C_{\finf}^{\sharp} /W,D(-Z))
$$ 
for each $i$ by $\sigma$-linearity of the composed map $F\circ\iota:D
\rightarrow \
D$, where $\iota:D \rightarrow  \sigma^*D$ is the map sending $x\to 1\otimes x$. 
\item 
We denote ${\cS}_D$ the mapping fiber of the map 
$${\bf 1}-\varphi:Ru_*D^0(-Z) \lra Ru_*D(-Z).$$ 
This complex is an object in the derived category of 
complexes of sheaves over $(C_{\finf })_{\mrm{\acute{e}t}}$ 
and we have a distinguished triangle: 
\begin{equation}\label{equation:syn}
\cS_D\lra Ru_{\ast} D^0(-Z) 
\overset{1-\varphi}{\lra} Ru_{\ast}D(-Z).
\end{equation}
We denote by $H^i_{\mrm{syn}} (C_{\finf},\mathcal{S}_D)$ 
the i-th cohomology group of $R\Gamma(C_{\finf} ,\cS_D$) and by 
$H^i_{\mrm{syn}} (C_{\finf},\mc{S}_{D} \otimes \Q_p /\Z_p )$  
the i-th cohomology group of $R\Gamma(C_{\finf},\cS_D)\otimes_{\Z_p}^{\mbb{L}} 
\Q_p/\Z_p$ by abuse of notation. 
\item 
Finally, let $X(D/\finf)$ denote the Pontrjagin dual of 
$H^i_{\mrm{syn}} (C_{\finf},\mc{S}_{D} \otimes \Q_p /\Z_p )$ 
and let $Y(D/\finf)=X(D/\finf)/X(D/\finf)(p).$
\end{enumerate}
In this paper, we are concerned with the two following crucial examples of Dieudonn\'e crystals:
\begin{example}\label{ex}
\begin{enumerate}
\item Let $A/F$ be an abelian variety satisfying the condition ${\bf (SS)}$. 
We denote $U$ the dense open subset of $C_{\finf}$ 
such that $A/F$ has good reduction on $U$ and 
$Z:=C_{\finf}\setminus U$. Let $\mc{A}$ 
be the N\'{e}ron model of $A\times_{F}\finf $ over $C_{\finf}$. 
It has been shown in \cite[\S 4.9]{KT} that the classical 
Dieudonn\'{e} crystal $D(\mc{A}_{/U})$ extends to a Dieudonn\'e 
crystal $D(\mc{A})$ over $C^{\sharp}_{\finf}/W$.
\item For any $p$-divisible group $\mc{G}$ over $C_{\finf}$, 
we can associate the classical Dieudonn\'{e} crystal $D({\cal G})$ 
which is a crystal over $C_{\finf}/W$. 
\end{enumerate}
\end{example}
In the rest of this section, we are going to deduce 
``finiteness theorems" over $F^{(p)}_{\infty}$ from 
``finiteness theorems" over $F_{\infty}$ introduced above. 
We denote by $\Gamma$ (resp. $\Gamma^{(p)}$, $\Gamma^{(p')}$) 
the Galois group $\mrm{Gal} (F_{\infty}/F)$ (resp. 
$\mrm{Gal} (F^{(p)}_{\infty}/F)$, 
$\mrm{Gal} (F_{\infty}/F^{(p)}_{\infty})$). 
Note that the group $\Gamma$ (resp. $\Gamma^{(p)}$, $\Gamma^{(p')}$) 
is isomorphic to $\what{\Z}$ (resp. $\Z_p$, $\uset{l\not= p}{\prod}{\Z_l}$). 
Note that $W=W(\oline{\mbb{F}}_q)$ is naturally endowed with a 
structure of $\La (\Gamma ) =\Z_p [[\Gamma ]]$-module. In fact 
$W(\mbb{F}_{q^n})$ is naturally a $\La (\Gamma )$-module through 
the action via a unique quotient $\Gamma_n$ of $\Gamma$ 
isomorphic to $\Z /(n)$. Hence $\varinjlim_n W(\mbb{F}_{q^n})$ is also 
a $\La (\Gamma )$-module. Since $W$ is the $p$-adic completion of 
$\varinjlim_n W(\mbb{F}_{q^n})$, we see that $W$ is a 
$\La (\Gamma )$-module. By crystalline base change theorem, 
We have: 
$$
H^i_{\mrm{crys}}(C_{\finf}^{\sharp} /W,E) \cong 
H^i_{\mrm{crys}}(C_{F}^{\sharp} /W(\mbb{F}_q),E) 
\otimes_{W(\mbb{F}_q)} W 
$$
for $E=D^0 (-Z)$ or $E=D(-Z)$. Hence 
$H^i_{\mrm{crys}}(C_{\finf}^{\sharp} /W,E)$ is naturally endowed with 
a structure of $\La (\Gamma ) $-module.
Since the triangle $(\ref{equation:syn})$ is also naturally obtained as 
the $p$-adic completion of the inductive limit of 
the similar triangles over $W(\mbb{F}_{q^n} )$, we also show that 
the syntomic cohomology is endowed with 
a structure of $\La (\Gamma ) $-module and the long exact sequence induced 
from the triangle $(\ref{equation:syn})$ is $\La (\Gamma ) $-linear. 
\par 
Though the Iwasawa algebra $\Lambda (\Gamma )$ is neither integral nor 
Noetherian, the algebra $\Lambda (\Gamma^{(p)} )$ is an integral Noetherian 
domain. Thus, we recall the following fundamental facts without proof: 
\begin{lem}\label{claim:extension}
Let $M'\lra M \lra M''$ be an exact sequence of 
$\La (\Gamma^{(p)})$-modules. Then, if two of these three modules 
are torsion of finite type over $\La (\Gamma^{(p)})$, the other one 
is also torsion of finite type over $\La (\Gamma^{(p)})$. 
In other words, the category of finite type torsion 
$\La (\Gamma^{(p)})$-modules 
are stable under taking a submodule, a quotient and an extension 
in the category of $\La (\Gamma^{(p)})$-modules. 
\end{lem} 
We use the following proposition for the proof 
of Theorem $\ref{thm:cyc}$ at the end of this section$:$ 
\begin{prop}\label{theorem-section2} 
 Let $D$ be a Dieudonn\'e crystal over $C^{\sharp}_{\finf}/W$. 
Then $X(D/\finf)_{\Gamma^{(p')}}$ is a finitely generated torsion 
$\La (\Gamma^{(p)})$-module. 
\end{prop}
Recall that the following exact sequence:
\begin{multline*}
\cdots \lra H^1_{\mrm{syn}} (C_{\finf},\mc{S}_{D}) \otimes \Q_p 
\buildrel{\alpha}\over\lra H^1_{\mrm{syn}} (C_{\finf},\mc{S}_{D} 
\otimes \Q_p/\Z_p) \\ 
\buildrel{\beta}\over\lra 
H^2_{\mrm{syn}} (C_{\finf},\mc{S}_{D})\buildrel{\gamma}\over\lra 
H^2_{\mrm{syn}} (C_{\finf},\mc{S}_{D})\otimes \Q_p\lra \cdots 
\end{multline*}
induces a short exact sequence 
\begin{equation}\label{exact-section2}
0\lra \mrm{Im}(\alpha)\lra H^1_{\mrm{syn}} (C_{\finf},\mc{S}_{D} 
\otimes \Q_p/\Z_p)\lra \mrm{Im}(\beta)\lra 0.
\end{equation}
By taking the Pontrjagin dual of $(\ref{exact-section2})$, we have 
\begin{equation}\label{exactbisbis-section2}
0\lra \mrm{Im}(\beta)^{\vee} \lra X(D/\finf) 
\lra \mrm{Im}(\alpha )^{\vee} \lra 0.
\end{equation}
Now, note that taking the $\Gamma^{(p')}$-coinvariant 
of a module over $\La (\Gamma ) \cong \La (\Gamma^{(p)} ) 
\what{\otimes}_{\Z_p} \La (\Gamma^{(p')} ) \cong \La (\Gamma^{(p)} )
[[\Gamma^{(p')}]]$ is equal to taking the base extension of the module 
by $\otimes_{\La (\Gamma )} \La (\Gamma^{(p)})$. 
Since taking the $\Gamma^{(p')}$-coinvariant is a left exact functor, 
the sequence $(\ref{exact-section2})$ gives us the following: 
\begin{equation}\label{exactbis-section2}
(\mrm{Im}(\beta)^{\vee}) _{\Gamma^{(p')}}\lra 
X(D/\finf)_{\Gamma^{(p')}}\lra 
(\mrm{Im}(\alpha)^{\vee})_{\Gamma^{(p')}} \lra 0, 
\end{equation}
where the modules and the morphisms are naturally defined 
over $\La (\Gamma^{(p)})$. 
Thanks to Lemma $\ref{claim:extension}$ and the short exact 
sequence $(\ref{exactbis-section2})$, 
it is enough to prove that 
$(\mrm{Im}(\alpha)^{\vee})_{\Gamma^{(p')}}$ and 
$(\mrm{Im}(\beta)^{\vee}) _{\Gamma^{(p')}}$ 
are torsion $\La (\Gamma^{(p)})$-modules of finite type 
in order to prove Proposition $\ref{theorem-section2}$. 
Thus, Proposition $\ref{theorem-section2}$ follows from 
Corollary $\ref{cor-section2}$ and Lemma $\ref{lem3-section2}$ which 
will be shown below. 
\begin{lem}\label{lem2-section2}
$H^1_{\mrm{syn}} (C_{\finf},\mc{S}_{D})\otimes \Q_p$ is a finite dimensional 
$\Q_p$-vector space.
\end{lem}
\begin{pf} The long exact sequence
\begin{multline*}
\cdots 
\lra H^i_{\mrm{syn}} (C_{\finf},\mc{S}_{D})\otimes \Q_p\lra 
H^i_{\mrm{crys}}(C_{\finf}^{\sharp} /W,D^0(-Z)))\otimes \Q_p \\ 
\buildrel{1-\varphi_i}\over\lra 
H^i_{\mrm{crys}}(C_{\finf}^{\sharp}/W,D(-Z))\otimes \Q_p\lra 
\cdots 
\end{multline*}
can be rewritten 
\begin{multline*}
\cdots \lra H^i_{\mrm{syn}} (C_{\finf},\mc{S}_{D})\otimes \Q_p\lra  
H^i_{\mrm{crys}}(C_{\finf}^{\sharp}/W,D(-Z))\otimes \Q_p \\ 
\buildrel{1-\varphi_i}\over\lra H^i_{\mrm{crys}}(C_{\finf}^{\sharp}/W,D(-Z))
\otimes \Q_p\lra \cdots 
\end{multline*}
thanks to Lemma $\ref{lem1-section2}$.\\ 
We deduce from this long exact sequence the following short exact sequence:
$$ 
0\lra \mrm{Coker}(1-\varphi_0)\to H^1_{\mrm{syn}} (C_{\finf},\mc{S}_{D})\otimes \Q_p\lra \mrm{Ker}(1-\varphi_1)\lra 0.
$$
Since $H^i_{\mrm{crys}}(C_{\finf}^{\sharp}/W,D(-Z))\otimes\Q_p$ 
is a finite dimensional $P_0$-vector space, with $P_0=\mrm{Frac}(W)$, 
the assertion results from the following classical result of which 
we omit the proof: 
\begin{sblem}\label{sblem-section2} Let $V$ be a finite dimensional 
$P_0$-vector space endowed with a $\sigma$-linear operator 
$\varphi:V\lra V$. Then 
$1-\varphi$ is a surjective map whose kernel is a 
finite dimensional $\Q_p$-vector space.
\end{sblem}
\end{pf}
\begin{cor}\label{cor-section2} 
The Pontrjagin dual $\mrm{Im}(\alpha)^{\lor}$ of $\mrm{Im}(\alpha)$ is a free 
$\Z_p$-module whose rank is equal to 
$\mrm{dim}_{\Q_p}(H^1_{\mrm{syn}} (C_{\finf},\mc{S}_{D})\otimes \Q_p)$.
In particular, $(\mrm{Im}(\alpha)^{\lor})_{\Gamma^{(p')}}$ is a 
finitely generated torsion $\La (\Gamma^{(p)})$-module. 
\end{cor}
\begin{proof}[Proof of Corollary $\ref{cor-section2}$]
By definition of the syntomic complex we have a long exact sequence
\begin{multline*}
\cdots \lra 
H^i_{\mrm{syn}} (C_{\finf},\mc{S}_{D}\otimes \Q_p/\Z_p)\lra 
H^i_{\mrm{crys}}(C_{\finf}^{\sharp}/W,D^0(-Z)\otimes \Q_p/\Z_p) \\ 
\buildrel{1-\varphi_i}\over\lra  H^i_{\mrm{crys}}(C_{\finf}^{\sharp}/W,D(-Z)
\otimes \Q_p/\Z_p)\lra \cdots 
\end{multline*}
where the middle and right handside modules are torsion. In particular, 
$H^1_{\mrm{syn}} (C_{\finf},\mc{S}_{D}\otimes \Q_p/\Z_p)$ is a torsion 
$\Z_p$-module. 
Thus, $\mrm{Im}(\alpha)$ is a torsion $\Z_p$-module 
which is a quotient of $H^1_{\mrm{syn}} (C_{\finf},\mc{S}_{D}) 
\otimes \Q_p$. We deduce from \ref{lem2-section2}, that $\mrm{Im}(\alpha)$ 
is cofree of finite rank $n\leq \mrm{dim}_{\Q_p}(H^1_{\mrm{syn}} 
(C_{\finf},\mc{S}_{D})\otimes \Q_p)$. 
But since $H^1_{\mrm{syn}} (C_{\finf},\mc{S}_{D})$ 
is an extension of a submodule of 
$H^1_{\mrm{crys}}(C_{\finf}^{\sharp}/W,D^0(-Z))$ by a quotient of 
$H^0_{\mrm{crys}}(C_{\finf}^{\sharp}/W,D(-Z))$, $\mrm{Ker}(\alpha)$ 
contains no non-zero $p$-divisible elements. 
This proves that the corank of $\mrm{Im}(\alpha)$ is equal to 
$\mrm{dim}_{\Q_p}(H^1_{\mrm{syn}} (C_{\finf},\mc{S}_{D})\otimes \Q_p)$. 
\end{proof} 
We now study the term $\mrm{Im}(\beta):$ 
\begin{lem}\label{lem3-section2}  
$(\mrm{Im}(\beta )^{\lor})_{\Gamma^{(p')}}$ is a 
finitely generated torsion $\La (\Gamma^{(p)})$-module.
\end{lem}
\begin{proof}[Proof of Lemma $\ref{lem3-section2}$] 
Note that we have an isomorphism $\mrm{Im}(\beta)\simeq \Ker(\gamma)$. 
The kernel of the map 
$\gamma:H^2_{\mrm{syn}} (C_{\finf},\mc{S}_{D})\lra 
H^2_{\mrm{syn}} (C_{\finf},\mc{S}_{D})\otimes \Q_p$ is 
$H^2_{\mrm{syn}} (C_{\finf},\mc{S}_{D})(p)$. 
Recall that we have a short exact sequence: 
$$ 
0\lra \mrm{Coker}(1-\varphi_1)\lra H^2_{\mrm{syn}} (C_{\finf},\mc{S}_{D})
\lra \mrm{Ker}(1-\varphi_2)\lra 0,
$$
By taking $p$-power torsion part of this sequence, we have 
\begin{equation}\label{exact2-section2}
0\lra \mrm{Coker}(1-\varphi_1)(p) \lra 
\mrm{Im}(\beta)
\lra \mrm{Ker}(1-\varphi_2) (p).
\end{equation} 
Taking the Pontrjagin duals and then taking the 
$\Gamma^{(p')}$-coinvariants of the modules, we have the following: 
\begin{equation}\label{exact2bisbis-section2}
(\mrm{Ker}(1-\varphi_2) (p)^{\vee})_{\Gamma^{(p')}} 
\lra (\mrm{Im}(\beta)^{\vee})_{\Gamma^{(p')}}  \lra 
(\mrm{Coker}(1-\varphi_1)(p) ^{\vee}) _{\Gamma^{(p')}} 
\lra 0,
\end{equation}
where the modules and the morphisms are defined 
over $\La (\Gamma^{(p)})$ as discussed around 
$(\ref{exactbis-section2})$. 
By the sequence $(\ref{exact2bisbis-section2})$ and Lemma 
$\ref{claim:extension}$, 
it is enough to show that 
$(\mrm{Coker}(1-\varphi_1)(p) ^{\vee}) _{\Gamma^{(p')}}$ and 
$(\mrm{Ker}(1-\varphi_2) (p)^{\vee})_{\Gamma^{(p')}}$ are 
finitely generated torsion 
$\La (\Gamma^{(p)})$-modules. 
\par
By the triangle $(\ref{equation:syn})$, 
$\mrm{Coker}(1-\varphi_1)$ 
(resp. $\mrm{Ker}(1-\varphi_2)$) is a quotient (resp. submodule) 
of $H^1_{\mrm{crys}}(C_{\finf}^{\sharp}/W,D(-Z))$ 
(resp. $H^2_{\mrm{crys}}(C_{\finf}^{\sharp}/W,D^0(-Z))$). 
By Lemma $\ref{lem1-section2}$, both 
$H^i_{\mrm{crys}}(C_{\finf}^{\sharp}/W,D^0(-Z))$ 
and $H^i_{\mrm{crys}}(C_{\finf}^{\sharp}/W,D(-Z))$ are 
  finitely generated $W(\oline{\F}_q)$-modules. 
  Hence, they are of the form 
  $W(\oline{\F}_q)^{\oplus r_0} \oplus 
  \bigoplus_{i=1}^r W_{n_i}(\oline{\F}_q)$, 
  where the $n_i\ge 1$ are not necessarily distinct. 
Now, we will terminate the proof admitting 
the following claim, which will be shown below: 
\begin{claim}\label{claim:final}
\begin{enumerate}
\item 
Let $M$ be a $\La (\Gamma )$-module such that 
$M \cong \oplus_{i=1}^r W_{n_i}(\oline{\F}_q)$, 
where the $n_i\ge 1$ are not necessarily distinct. 
Then 
\begin{enumerate}
\item 
For any $\La (\Gamma )$-linear quotient $N$ of $M$, 
$({N}^{\vee})_{\Gamma^{(p')}}$ is a finitely generated 
torsion $\La (\Gamma^{(p)} )$-module. 
\item 
For any $\La (\Gamma )$-submodule $N'$ of $M$, 
$({N'}^{\vee})_{\Gamma^{(p')}}$ is a finitely generated 
torsion $\La (\Gamma^{(p)} )$-module. 
\end{enumerate}
\item 
For any $\La (\Gamma )$-linear quotient $S$ of the 
$\La (\Gamma )$-module $W(\oline{\mathbb{F}}_q)$, 
$S (p)$ is equal to $S [p^t ]$ for a sufficiently large 
integer $t$. 
\end{enumerate}
\end{claim}
For $\mrm{Ker}(1-\varphi_2)(p)^\lor$, we have the sequence:
$$
H^2_{\mrm{crys}}(C_{\finf}^{\sharp}/W,D^0(-Z))(p)^\lor 
\lra \mrm{Ker}(1-\varphi_2)(p)^\lor \lra 0 .
$$
Taking the $\Gamma^{(p')}$-coinvariant, we have: 
$$
(H^2_{\mrm{crys}}(C_{\finf}^{\sharp}/W,D^0(-Z))(p)^\lor )_{\Gamma^{(p')}} 
\lra (\mrm{Ker}(1-\varphi_2)(p)^\lor )_{\Gamma^{(p')}} \lra 0 .
$$
Since $(H^2_{\mrm{crys}}(C_{\finf}^{\sharp}/W,D^0(-Z))(p)^\lor )
_{\Gamma^{(p')}} $ is a finitely generated 
torsion $\La (\Gamma^{(p)} )$-module by Claim $\ref{claim:final}.1$, 
$(\mrm{Ker}(1-\varphi_2)(p)^\lor )_{\Gamma^{(p')}}$ is 
a finitely generated torsion $\La (\Gamma^{(p)} )$-module by 
Lemma $\ref{claim:extension}$. 
For $\mrm{Coker}(1-\varphi_1)(p)^\lor$, we start from the following 
exact sequence of $\La (\Gamma )$-modules: 
$$
0 \lra 
\mrm{Im}(1-\varphi_1) \lra H^1_{\mrm{crys}}(C_{\finf}^{\sharp}/W,D(-Z)) 
\lra \mrm{Coker}(1-\varphi_1) \lra 0 .
$$
By Claim $\ref{claim:final}.2$ and by the remark on the structure on 
$H^1_{\mrm{crys}}(C_{\finf}^{\sharp}/W,D(-Z)) $ given before, 
there is a sufficiently large $s$ such that 
$\mrm{Im}(1-\varphi_1)(p)$ 
(resp. $H^1_{\mrm{crys}}(C_{\finf}^{\sharp}/W,D(-Z))(p)$, 
$\mrm{Coker}(1-\varphi_1)(p)$) is equal to 
$\mrm{Im}(1-\varphi_1)[p^s ]$ 
(resp. $H^1_{\mrm{crys}}(C_{\finf}^{\sharp}/W,D(-Z))[p^s ]$, 
$\mrm{Coker}(1-\varphi_1)[p^s ]$ ). 
Hence, we have the following sequence: 
\small 
\begin{equation}\label{equation:subquobyGp}
\left( \mrm{Im}(1-\varphi_1)/(p^s)\mrm{Im}(1-\varphi_1) \right)^\lor 
\overset{a}{\lra} \mrm{Coker}(1-\varphi_1)(p)^\lor 
\overset{b}{\lra} 
H^1_{\mrm{crys}}(C_{\finf}^{\sharp}/W,D(-Z))(p)^\lor 
\end{equation}
\normalsize 
As is explained above, 
the image of the map $b$ in $(\ref{equation:subquobyGp})$ 
is of the form $N^{\vee}$ for a $\La (\Gamma )$-linear quotient 
$N$ of $\oplus_{i=1}^r W_{n_i}(\oline{\F}_q)$. 
Since $\mrm{Im}(1-\varphi_1)/(p^s)\mrm{Im}(1-\varphi_1) $ is a 
$\La (\Gamma )$-linear quotient of 
$H^1_{\mrm{crys}}(C_{\finf}^{\sharp}/W,D^0(-Z))/(p^s) 
H^1_{\mrm{crys}}(C_{\finf}^{\sharp}/W,D^0(-Z))$, 
the image of the map $a$ is of the form ${N'}^{\vee}$ for a 
$\La (\Gamma )$-linear 
submodule $N'$ of $\oplus_{j=1}^{r'} W_{n'_j}(\oline{\F}_q)$. 
Since $\Gamma^{(p')}$ is pro-cyclic group, 
the sequence $(\ref{equation:subquobyGp})$ induces the following sequence: 
\begin{equation}\label{equation:subquobyGpbis} 
({N'}^{\vee})_{\Gamma^{(p')}} \lra 
(\mrm{Coker}(1-\varphi_1)(p)^\lor )_{\Gamma^{(p')}} \lra 
(N^{\vee})_{\Gamma^{(p')}} \lra 0 . 
\end{equation}
Since the modules on the right and the left are finitely generated 
torsion $\La (\Gamma^{(p)} )$-modules by Claim $\ref{claim:final}.1$, 
$(\mrm{Coker}(1-\varphi_1)(p)^\lor )_{\Gamma^{(p')}} $ is 
a finitely generated torsion $\La (\Gamma^{(p)} )$-module by 
Lemma $\ref{claim:extension}$. 
\par 
To finish the proof of Lemma $\ref{lem3-section2}$, we prove 
Claim $\ref{claim:final}$. 
Since $W_m (\oline{\mbb{F}}_q)$ is a successive extension of 
$W_1 (\oline{\mbb{F}}_q)=\oline{\mbb{F}}_q$, it is enough to 
show the first assertion only 
for $M=W_1 (\oline{\mbb{F}}_q)=\oline{\mbb{F}}_q$. 
Let $U$ be an open subgroup of $\Gamma$, $(\oline{\mbb{F}}_q)^{U} $ 
is a finite extension of $\mathbb{F}_q$. We have 
a natural $\La (\Gamma )$-linear isomorphism 
$(\oline{\mbb{F}}_q)^{U} \cong \mbb{F}_q [(\Gamma /U )^{\vee}]$. 
If $\Gamma /U$ is sufficiently small so that the order of $\Gamma /U $ 
divides $q-1$, we can take a basis $\{x_1 ,\cdots ,x_l \}$ of 
$(\oline{\mbb{F}}_q)^{U}$ over 
$\mbb{F}_q$, so that the action of $\Gamma$ on $(\oline{\mbb{F}}_q)^{U}$ 
is represented by a diagonal matrix. 
For each member $x_i$ in the above set of basis, 
$g\mapsto (x_i )^g /x_i $ gives a character of $\Gamma$. 
Every character of $\Gamma $ factored by $\Gamma /U$ 
is given this way and for different $x_i ,x_j$ the associated 
characters are different. This explains the canonical isomorphism 
$(\oline{\mbb{F}}_q)^{U} \cong \mbb{F}_q [(\Gamma /U )^{\vee}]$ 
for these special $U$'s. 
Since they are rather elementary, we do not give 
further explanation nor the proof on the above isomorphisms. 
By taking an inductive limit of 
$(\oline{\mbb{F}}_q)^{U} \cong \mbb{F}_q [(\Gamma /U )^{\vee}]$ with 
respect to open subgroups $U$ of $\Gamma$, we have: 
\begin{equation}\label{equation:decomposition} 
M \cong \mbb{F}_q [ (\Gamma )^{\vee}] \cong 
(\mbb{F}_p [ (\Gamma )^{\vee}])^{\oplus \mrm{ord}_p (q)}.  
\end{equation}  
\par 
Hence, it suffices to show the first assertion only 
for a $\La (\Gamma )$-linear quotient $N$ of 
$\mbb{F}_p [ (\Gamma )^{\vee}]$. 
Taking Pontrjagin dual, $N^{\vee}$ is a $\La (\Gamma )$-submodule of 
$(\mbb{F}_p [ (\Gamma )^{\vee}])^{\vee} \cong \mbb{F}_p [[ \Gamma ]] 
\cong \La (\Gamma )/(p)$. Hence we have $N^{\vee} =I$ 
for an ideal $I$ of $\mbb{F}_p [[ \Gamma ]]$. 
Since $\Gamma^{(p')}$is procyclic, the short exact sequence: 
$$
0 \lra I \lra \mbb{F}_p [[ \Gamma ]] \lra \mbb{F}_p [[ \Gamma ]] /I 
\lra 0 
$$
induces the following sequence: 
\begin{equation}\label{equation:lastclaim}
(\mbb{F}_p [[ \Gamma ]] /I )^{\Gamma^{(p')}} 
\lra (N^{\vee})_{\Gamma^{(p')}} \lra 
(\mbb{F}_p [[ \Gamma ]])_{\Gamma^{(p')}} \lra 
(\mbb{F}_p [[ \Gamma ]] /I )_{\Gamma^{(p')}}
\lra 0 . 
\end{equation}
The third term $(\mbb{F}_p [[ \Gamma ]])_{\Gamma^{(p')}}$ is isomorphic to 
$\mbb{F}_p [[ \Gamma^{(p)} ]] \cong \La (\Gamma^{(p)})/(p)$. 
For the first term, we decompose as $\mbb{F}_p [[ \Gamma ]] /I 
\cong \mbb{F}_p [[ \Gamma^{(p)} ]][[\Gamma^{(p')}]] /I 
\cong \left( \left( 
(\mbb{F}_p [[ \Gamma^{(p)} ]]/I_0 )\right) 
[[\Gamma^{(p')}]] \right) /\oline{I} $ 
where $I_0 =I \cap \mbb{F}_p [[ \Gamma^{(p)} ]]$ and $\oline{I}$ is 
the image of $I$ via $\mbb{F}_p [[ \Gamma ]] 
\cong \mbb{F}_p [[ \Gamma^{(p)} ]][[\Gamma^{(p')}]] 
\twoheadrightarrow 
\left( \mbb{F}_p [[ \Gamma^{(p)} ]] /I_0 \right) [[\Gamma^{(p')}]]$. 
Now, it is easy to see that 
$(R [[ \Gamma^{(p')} ]] /\oline{I})^{\Gamma^{(p')}} = 
R$ when $R$ is a pro-$p$ algebra invariant under $\Gamma^{(p')}$ 
and $\oline{I}$ is the ideal of $R [[\Gamma^{(p')}]]$ 
such that $R \cap \oline{I} =0$. Hence, the first term of 
$(\ref{equation:lastclaim})$ 
is isomorphic to $\mbb{F}_p [[ \Gamma^{(p)} ]]/I_0$, which is a finitely 
generated torsion $\La (\Gamma^{(p)})$-module. 
Since $\La (\Gamma^{(p)} )$ is noetherian, the assertion $(\mrm{a})$ 
is an immediate consequence of the sequence $(\ref{equation:lastclaim})$. 
\par 
For the assertion $(\mrm{b})$, by the same argument as above, 
it is sufficient to prove when $N'$ is a $\La (\Gamma )$-linear 
submodule of $\mbb{F}_p [ (\Gamma )^{\vee}]$. Then we have 
a $\La (\Gamma^{(p)})$-linear map 
$\La (\Gamma^{(p)} )/(p) \twoheadrightarrow ({N'}^{\vee})_{\Gamma^{(p')}}$. 
This completes the proof of $(\mrm{b})$. 
\par 
By the similar argument as above, we have $W(\oline{\mbb{F}}_q) 
\cong 
\left( \varprojlim_n \left( \Z /(p^n) [(\Gamma )^{\vee}] \right) 
\right)^{\oplus \mrm{ord}_p (q)}$. It is not difficult to show 
that only non-trivial 
$\La (\Gamma)$-linear quotients of 
$\varprojlim_n \left( \Z /(p^n) [(\Gamma )^{\vee}] \right)$ is either 
$\varprojlim_n \left( \Z /(p^n) [(\Gamma )^{\vee}] \right)$ itself 
or $ \Z /(p^t) [(\Gamma )^{\vee}] $ for some $t$. 
This suffices for the second assertion. 
This completes the proof of Claim $\ref{claim:final}$ (and hence completes 
the proof of Lemma $\ref{lem3-section2}$). 
\end{proof}
Before starting the proof of Theorem $\ref{thm:cyc}$, we recall 
the following result: 
\begin{prop}\label{prop-section2} 
\begin{enumerate}
\item Let $A/F$ an abelian variety satisfying the condition ${\bf (SS)}$.
Let $D(\mc{A})$ be the Dieudonn\'{e} crystal over $C_{\finf^{(p)}}^{\sharp}/W$ 
associated to the N\'{e}ron model $\mc{A}$ of $A\times_{F}\finf^{(p)} $. 
Then we have a monomorphism of $\La (\Gamma^{(p)})$-modules 
$$X(A/\finf^{(p)} )\lra X(D(\mc{A})/\finf^{(p)} ).$$
\item Let ${\cal G}$ be a p-divisible group over $C_{\finf^{(p)}}$, 
then $X(D({\cal G})/\finf^{(p)})$ 
is isomorphic to the Pontrjagin dual of 
$H^1_{\mrm{fl}}(C_{\finf^{(p)}},{\cal G})$. 
\end{enumerate}
\end{prop}
\begin{pf} The first assertion is proved as in \cite[\S 2.5, 5.13]{KT}. 
The second assertion is reduced to \cite[\S 5.10]{KT}.
\end{pf}
\begin{proof}[Proof of Theorem $\ref{thm:cyc}$]
By Lemma $\ref{lem:ssred}$, we can assume that $A/F$ has 
semi-stable reduction. 
Note that $\mrm{Gal}(\finf/\fpinf)$ is isomorphic to 
$\Gamma^{(p')}\cong \uset{l\not= p}{\prod}{\Z_l}$. 
Let $\mbb{F}^{(p)}_q = F^{(p)}_{\infty} \cap \oline{\mbb{F}}_q$. 
We have the following commutative diagram: 
\small 
$$
\begin{CD}
0 @>>> \mrm{Ker}(1-\varphi^{(p)}_1) @>>> 
H^1_{\mrm{crys}}(C_{\finf^{(p)}}^{\sharp}/W(\mbb{F}^{(p)}_q) ) 
@>{1-\varphi^{(p)}_1 }>> 
H^1_{\mrm{crys}}(C_{\finf^{(p)}}^{\sharp}/W(\mbb{F}^{(p)}_q) )' 
\\ 
@. @V{a}VV @V{b}VV @VV{c}V \\
0 @>>> \mrm{Ker}(1-\varphi_1)^{\Gamma^{(p')}} @>>> 
H^1_{\mrm{crys}}(C_{\finf}^{\sharp} /W )
^{\Gamma^{(p')}} 
@>>{1-\varphi_1 }> 
{H^1_{\mrm{crys}}(C_{\finf}^{\sharp} /W)'}
^{\Gamma^{(p')}},  \\ 
\end{CD}
$$
\normalsize 
where $H^1_{\mrm{crys}}(C_{\finf^{(p)}}/W(\mbb{F}^{(p)}_q) )$ (resp. 
$H^1_{\mrm{crys}}(C_{\finf^{(p)}}/W(\mbb{F}^{(p)}_q)  )'$) means 
$H^1_{\mrm{crys}}(C_{\finf^{(p)}}^{\sharp} /W(\mbb{F}^{(p)}_q ),D^0 (-Z) 
\otimes {\Q_p/\Z_p})$ 
(resp. $H^1_{\mrm{crys}}(C_{\finf^{(p)}}^{\sharp} /W(\mbb{F}^{(p)}_q )
,D (-Z) \otimes {\Q_p/\Z_p} ) $) and we take the similar definitions 
for $\finf$. 
We denote by $\varphi^{(p)}_i$ the frobenius operator 
over $F^{(p)}_{\infty}$. 
The vertical maps $b$ and $c$ in the diagram are isomorphism 
by crystalline base change theorem. 
Thus, the map $a$ is isomorphism. 
By the triangle $(\ref{equation:syn})$, 
we also have the following diagram: 
\small 
$$
\begin{CD}
0 @>>> \mrm{Coker}(1-\varphi^{(p)}_0) @>>> 
H^1_{\mrm{syn}} (C_{\finf^{(p)}},\mc{S}_{D} 
\otimes {\Q_p/\Z_p})
@>>> 
\mrm{Ker}(1-\varphi^{(p)}_1) @>>> 0
\\ 
@. @V{d}VV @VVV @VV{a}V @.\\
0 @>>> \mrm{Coker}(1-\varphi_0)^{\Gamma^{(p')}} @>>> 
H^1_{\mrm{syn}} (C_{\finf}, \mc{S}_{D} 
\otimes {\Q_p/\Z_p})^{\Gamma^{(p')}}
@>>> 
\mrm{Ker}(1-\varphi_1)^{\Gamma^{(p')}} @.   \\ 
\end{CD}
$$
\normalsize
We have the following claim: 
\begin{claim}\label{claim:aaa}
The Pontrjagin dual of the cokernel of the map: 
$$1 - \varphi^{(p)}_0 : 
H^0_{\mrm{crys}}(C_{\finf^{(p)}}^{\sharp} /W(\mbb{F}^{(p)}_q ),
D^0 (-Z) \otimes \Q_p /\Z_p)  
\lra H^0_{\mrm{crys}}(C_{\finf^{(p)}}^{\sharp} /W(\mbb{F}^{(p)}_q ),D(-Z) 
\otimes \Q_p /\Z_p)$$ is a finitely generated 
torsion $\La (\Gamma^{(p)})$-module. 
\end{claim}
By this claim, the Pontrjagin dual of $\mrm{Coker}(1 - \varphi^{(p)}_0)$ 
is a finitely generated torsion 
$\La (\Gamma^{(p)} )$-module. 
Taking the Pontrjagin dual of the above commutative diagram, 
we have a $\La (\Gamma^{(p)})$-linear map 
$X(D(\mc{A})/\finf)_{\Gamma^{(p')}} \lra X(D(\mc{A})/\fpinf)$ 
whose cokernel is a torsion $\La (\Gamma^{(p)} )$-module 
of finite type. 
On the other hand, 
$X(D(\mc{A})/\finf)_{\Gamma^{(p')}}$ is a finitely generated 
torsion $\La (\Gamma^{(p)})$-module by 
Proposition $\ref{theorem-section2}$. 
Finally we will prove Claim $\ref{claim:aaa}$. 
This will complete the proof of Theorem $\ref{thm:cyc}$ by 
Proposition $\ref{prop-section2}$. 
We consider the following commutative diagram whose vertical maps are 
all $1-\varphi^{(p)}_i$: 
\small 
\begin{equation}\label{equation:crystor1}
\begin{CD}
H^0_{\mrm{crys}}( \Z_p  ) @>>> H^0_{\mrm{crys}}( \Q_p  ) @>>> 
H^0_{\mrm{crys}}( \Q_p /\Z_p  ) 
@>>> 
H^1_{\mrm{crys}}( \Z_p )_{\mrm{tor}} @>>> 0 
\\ 
@VVV @VVV @VV{1-\varphi^{(p)}_0}V @VVV @. \\
H^0_{\mrm{crys}}( \Z_p  )' @>>> H^0_{\mrm{crys}}( \Q_p  )' @>>> 
H^0_{\mrm{crys}}( \Q_p /\Z_p  )' 
@>>> 
H^1_{\mrm{crys}}( \Z_p )'_{\mrm{tor}} @>>> 0 ,  \\ 
\end{CD}
\end{equation}
\normalsize 
where $H^i_{\mrm{crys}} (\ast )$ 
in the upper line with $\ast =\Z_p ,\Q_p ,\Q_p /\Z_p$ 
means $H^i_{\mrm{crys}}(C_{\finf^{(p)}}^{\sharp} /W(\mbb{F}^{(p)}_q ),
D^0(-Z) \otimes \ast )$, $H^i_{\mrm{crys}} (\ast )'$ in the lower 
line means $H^i_{\mrm{crys}}(C_{\finf^{(p)}}^{\sharp} /W(\mbb{F}^{(p)}_q ),
D(-Z) \otimes \ast )$ and $(\ \ )_{\mrm{tor}}$ is the $\Z_p$-torsion part. 
We can show as in Lemma $\ref{lem1-section2}$ (a) that the map 
$1: H^0_{\mrm{crys}}( \Q_p  ) \lra H^0_{\mrm{crys}}( \Q_p  ) '$ induces 
a natural isomorphism 
of finite dimensional $W(\mbb{F}^{(p)}_q ) \otimes \Q_p$-vector spaces. 
Since $\varphi^{(p)}_0$ is $\sigma$-linear, 
$1-\varphi^{(p)}_0: H^0_{\mrm{crys}}( \Q_p  ) \lra H^0_{\mrm{crys}}( \Q_p  )'$ 
is surjective (see \cite[Lemma 6.2]{EL} for a similar argument). 
By the diagram $(\ref{equation:crystor1})$ and 
by the snake lemma, $\mrm{Coker}[1-\varphi^{(p)}_0 : 
H^0_{\mrm{crys}}( \Q_p /\Z_p  ) \rightarrow H^0_{\mrm{crys}}( \Q_p /\Z_p  )']$ 
is isomorphic to $\mrm{Coker}[1-\varphi^{(p)}_1 : 
H^1_{\mrm{crys}}(\Z_p  )_{\mrm{tor}} \rightarrow 
H^1_{\mrm{crys}}( \Z_p  )'_{\mrm{tor}}]$. 
Thus, $\mrm{Coker}[1-\varphi^{(p)}_0 : 
H^0_{\mrm{crys}}( \Q_p /\Z_p  ) \rightarrow H^0_{\mrm{crys}}
( \Q_p /\Z_p  )']$ is a $\La (\Gamma^{(p)})$-linear quotient of 
a finitely generated torsion $W(\mbb{F}^{(p)}_q )$-module 
$H^1_{\mrm{crys}}( \Z_p  )'_{\mrm{tor}}$. 
This completes the proof of Claim $\ref{claim:aaa}$.
\end{proof}
\section{$\mu$-invariants of $X (A /\fpinf)$} 
We showed that $X (A /\fpinf)$ is a torsion 
$\La (\Gamma^{(p)})$-module for an abelian variety $A$ over $F$. 
In this section, we will prove that the $\mu$-invariant of 
$X (A /\fpinf)$ is zero under certain assumptions 
(Theorem $\ref{thm:cycmu}$). 
\par 
Before giving the proof, we prepare several fundamental facts 
which will be necessary in the proof. 
\begin{lem}\label{lem:goodred}
Let $K$ be a finitely generated one-variable function field 
over a field $\mbb{F}$ contained in $\oline{\mbb{F}}_q$ 
and let $A$ be an abelian variety over $K$. 
Then, the Selmer group $\sel (A/K)$ is isomorphic 
to the flat cohomology $H^1_{\mrm{fl}}(C_K, \mc{A}\{ p\})$ if 
$A$ has everywhere good reduction over $K$, where $\mc{A}$ is the 
model of $A$ over $C_K$. 
Here, $C_K$ denotes the projective smooth curve which is 
the model of a function field $K$.
\end{lem}
Let us recall the following claim: 
\begin{claim}\label{lem:localization} 
Let $V$ be a smooth curve over a finite field $\mathbb{F}_q$ 
and let $U \subset V$ be a dense open subset. 
For any finite flat group scheme $\mc{F}$ 
over $\mathcal{O}_V$, we have the following natural exact sequence:
$$
0 \lra \Hfl (V ,\mc{F}) \lra \Hfl (U,\mc{F} ) \lra \uset{v \in V \setminus U}
{\bigoplus} \dfrac{\Hfl (K_v ,F)}{\Hfl (\mathcal{O}_v ,\mc{F})},
$$
where $\mathcal{O}_v$ is the completion of $\mathcal{O}_V$ at 
the closed point $v\in V$, $K_v$ is the field of fraction of 
$\mathcal{O}_v$ and $F$ is the generic fiber of $F$. 
\end{claim}
Though this claim seems to be known to the experts, 
we will give a short sketch of the proof by recalling fundamental facts 
on flat cohomologies. For fundamental tools and facts used in the proof, 
the reader can refer to \cite[Chapter III]{Mi}.
\begin{proof}[Proof of Claim $\ref{lem:localization}$] 
Let $Z=V\setminus U$ be a closed subscheme of $V$. 
We have a usual localization sequence:
\begin{equation}\label{equation:milne}
H^1_{Z} (V ,F) \lra 
\Hfl (V,\mc{F} )\lra \Hfl (U,\mc{F}) \lra H^2_{Z} (V ,\mc{F}), 
\end{equation}
where $H^i_Z (V,\mc{F})$ is a flat cohomology with support in $Z$. 
By \cite[III, \S$7$]{Mi}, we have:
$$ 
\begin{cases}
& H^1_Z (V,\mc{F}) = 0 \\
& H^2_Z (V,\mc{F}) = \underset{v\in V\setminus U}{\bigoplus} 
H^2_v (V,\mc{F}) \cong \underset{v\in V\setminus U}{\bigoplus} 
\dfrac{\Hfl (K_v ,F)}{\Hfl (\mathcal{O}_v ,\mc{F})}
\end{cases}
$$
This completes the proof. 
\end{proof}
\begin{proof}[Proof of Lemma $\ref{lem:goodred}$]
By Claim $\ref{lem:localization}$, we have: 
$$
0 \lra H^1_{\mrm{fl}} (C_K,\mc{A}[p^m]) \lra H^1_{\mrm{fl}} (C_K 
\setminus Z ,\mc{A}[p^m]) 
\lra \underset{v\in Z}{\oplus}
\dfrac{H^1_{\mrm{fl}} (K_v ,A[p^m])}
{H^1_{\mrm{fl}} (\mathcal{O}_v ,\mc{A}[p^m])}
$$
By taking inductive limit with respect to $Z$ and $m$, we have: 
$$
0 \lra H^1_{\mrm{fl}} (C_K,\mc{A}\{ p\}) \lra H^1_{\mrm{fl}} (K,A\{ p\}) 
\lra \underset{v\in C_K}{\oplus}
\dfrac{H^1_{\mrm{fl}} (K_v ,A\{ p\})}{H^1_{\mrm{fl}} (\mathcal{O}_v ,\mc{A}
\{ p\})}.
$$
On the other hand, we easily show that  
$\dfrac{H^1_{\mrm{fl}} (K_v ,A\{ p\})}{H^1_{\mrm{fl}} 
(\mathcal{O}_v ,\mc{A}\{ p\})} \cong H^1 (K_v ,A)(p)$ 
since $H^1_{\mrm{fl}} (\mathcal{O}_v ,\mc{A})\cong 
H^1_{\mrm{fl}} (\mbb{F}_v ,\oline{\mc{A}}_v )=0$ by \cite{Lang},  
where $\mbb{F}_v$ is the residue field of $\mathcal{O}_v$ 
and $\oline{\mc{A}}_v$ is the special fiber of $\mc{A}$ at 
$\mrm{Spec}(\mbb{F}_v)$. 
This completes the proof of Lemma $\ref{lem:goodred}$ 
by comparing with the exact sequence $(\ref{equation:seldef})$
which defines the Selmer group. 
\end{proof}
By using the following lemma, we reduce the proof of Theorem 
$\ref{thm:cycmu}$ into a simplified situation . 
\begin{lem}\label{pro:mainiso}
Let $F'/F$ be a finite Galois extension. 
We denote by $A'$ the extension $A'=A \times_F F'$.
If $\mu (X(A'/\Fpinf ))=0$, then we have $\mu (X(A/\fpinf ))=0$. 
\end{lem}
\begin{proof}
The assertion that $\mu (X(A/\fpinf ))=0$ (resp. $\mu (X(A'/\Fpinf ))=0$)
is equivalent to the assertion that $X(A/\fpinf )$ (resp. $X(A'/\Fpinf )$) 
is a finitely generated $\Z_p$-module. 
As the proof of Lemma $\ref{lem:ssred}$, we have a natural 
homomorphism $X(A'/\Fpinf )_J \lra X(A/\fpinf )$ whose cokernel 
is finite (Here, $J$ is $\mrm{Gal}(\Fpinf /\fpinf)$). 
Hence we complete the proof. 
\end{proof}
By Lemma $\ref{pro:mainiso}$, 
we may (and we will) assume that $A/F$ is isomorphic to an ordinary 
abelian variety or a supersingular abelian variety 
defined over a finite field in order to prove Theorem $\ref{thm:cycmu}$. 
First, we consider the case where the following condition is satisfied 
through this section for the proof of Theorem $\ref{thm:cycmu}$:  
\vspace*{5pt}
\\ 
{\bf (OF)} $A/F$ is isomorphic to an ordinary abelian variety 
defined over a finite field $\mathbb{F}_q \subset F$. 
\vspace*{5pt}
\\  
Note that $\finf$ is a Galois extension of 
$\fpinf$ with Galois group $\Gamma ^{(p')}\cong 
\uset{l\not= p}{\prod}{\Z_l}$. 
As in the argument of the proof of Theorem $\ref{thm:cyc}$, we have:  
$$
\Hfl (C_{\fpinf} ,A\{ p\})^{\lor} 
\cong (H^1_{\mrm{fl}}(C_{\finf} ,A\{ p\})^{\lor})
_{\Gamma ^{(p')}} .
$$
Thus we have reduced ``finiteness" over $\fpinf$ to 
Lemma $\ref{lem:Finf}$ (``finiteness" over $\finf$) below.
\begin{lem}\label{lem:Finf}
Let us assume the condition ${\bf (OF)}$ for $A/F$. 
Then, $\Hfl (C_{\finf},A\{ p\})^{\lor}$ is 
a finitely generated $\Z_p$-modules.
\end {lem}
\begin{pf} 
Since $A$ is an abelian variety with good ordinary reduction, 
the connected part (resp. \'{e}tale part) $A\{ p\}^{\mrm{conn}}$ 
(resp. $A\{ p\}^{\text{\'{e}t}}$) of $A\{ p\}$ has rank $g=\mrm{dim}
(A)$. Note that a finite flat group scheme $A \{ p\}^{\mrm{conn}}$ 
(resp. $A \{ p\}^{\text{\'{e}t}}$) over $C_{\finf}$ is a constant 
scheme defined over the base $\oline{\mathbb{F}}_p$ of the curve 
$C_{\finf}$. Since 
$A \{ p\}^{\mrm{conn}}$ (resp. $A \{ p\}^{\text{\'{e}t}}$) 
is isomorphic to $(\mu_{p^{\infty}})^g$ (resp. $(\Q_p /\Z_p)^g$) over 
$\oline{\mathbb{F}}_p$, we have the following exact sequence: 
$$
H^1_{\mrm{fl}} (C_{\finf} ,(\mu_{p^{\infty}})^g) 
\lra H^1_{\mrm{fl}} (C_{\finf} ,A\{ p\}) 
\lra H^1_{\mrm{fl}} (C_{\finf} ,(\Q_p /\Z_p)^g). 
$$ 
Thus we reduce the proof of Lemma $\ref{lem:Finf}$ to the following claim:
\begin{claim}\label{lem:wakeru}
$\Hfl (C_{\finf},\mu_{p^{\infty}})^{\lor}$ and 
$\Hfl (C_{\finf},\Q_p /\Z_p)^{\lor}$ are 
free $\Z_p$-modules of finite rank.
\end {claim}
Let us prove the claim in the rest of the proof. 
We first study 
$\Hfl (C_{\finf},\mu_{p^{\infty}})^{\lor}$. 
For a scheme $V$, \cite[Proposition 3.7]{Mi1} gives us 
an equality $H^1_{\mrm{fl}}(V ,\mc{O}_{V}^{\times})\cong 
\mrm{Pic}(V)$. 
On the other hand, 
we have an exact sequence: 
$$
0\lra \mu_{p^n} \lra \mc{O}^{\times}_V 
\overset{\times p^n}{\lra} \mathcal{O}^{\times}_V \lra 0 
$$  
on the flat site over $V$. 
This implies that 
$\Hfl (\mrm{Spec}(R) ,\mu_{p^n}) = R^{\times}/( R^{\times})^{p^n}$ 
for a local ring $R$. 
By Lemma $\ref{lem:localization}$, 
we have an exact sequence:
$$
0\lra \Hfl (C_{\finf}, \mu_{p^n} )\lra 
(\finf)^\times/(\finf)^\times)^{p^n}
\lra \underset{v\in C_{\finf} 
\setminus \mathrm{Spec} (\finf )}{\bigoplus} 
\dfrac{(F_{\infty ,v})^\times/((F_{\infty ,v})^\times)^{p^n}}
{\Cal{O}_v^\times/(\Cal{O}_v^\times)^{p^n}},
$$
where $\Cal{O}$ is the ring of integers of $\finf $. 
Note that we have$:$ 
$$
\dfrac{(F_{\infty ,v})^\times/((F_{\infty ,v})^\times)^{p^n}}
{\Cal{O}_v^\times/(\Cal{O}_v^\times)^{p^n}} 
\cong 
\dfrac{(F_{\infty ,v})^\times/\Cal{O}_v^\times} 
{\bigl( (F_{\infty ,v})^\times/\Cal{O}_v^\times 
\bigr)^{p^n}} 
\cong 
\Z /(p^n).
$$
Let us consider the following commutative diagram$:$ 
$$
\begin{CD}
0 @>>> (\finf )^{\times} @>{\times p^n}>> 
(\finf )^{\times} 
@>>> (\finf )^{\times}/ ((\finf )^{\times} )^{p^n} 
@>>> 0 \\
@. @V{\oplus \mrm{val}_v}VV @VV{\oplus \mrm{val}_v}V 
@VVV @. \\
0 @>>> \underset{v}{\bigoplus}  
\Z @>{\times p^n}>> 
\underset{v}{\bigoplus}  
\Z @>>> \underset{v}{\bigoplus}  
\Z /p^n \Z @>>> 0, 
\end{CD}
$$
where $v$ runs the primes of $v\in C_{\finf}
\setminus \mathrm{Spec}(\finf )$ for the summations 
in the lower line. 
The kernel of the middle vertical map is equal to 
$(\oline{\mbb{F}}_p )^{\times}$. 
By applying the snake lemma to the above diagram, 
we get the following exact sequence
$$0 \lra  (\oline{\mbb{F}}_p )^{\times} / ((\oline{\mbb{F}}_p )^{\times})^{p^n}
\lra \Hfl (C_{\finf},\mu_{p^n}) \to \mrm{Cl} (\finf ) 
[p^n] \lra  0.$$
Since $(\oline{\mbb{F}}_p )^{\times} / ((\oline{\mbb{F}}_p )^{\times})^{p^n}$ 
is trivial, we have $\Hfl (C_{\finf},\mu_{p^{\infty}}) 
\cong \mrm{Cl}(\finf ) \{ p \}$. 
By [W], we have an isomorphism 
$$
\mrm{Cl}(\finf ) \{ p \}\simeq J_{\finf }
(\oline{\mbb{F}}_p)\{ p \}$$
where $J_{\finf }$ is the Jacobian variety 
associated to $C_{\finf}$. 
However, by [Mu], we have $J_{\finf}(\oline{\mbb{F}}_p)
\{p\}\simeq (\Q_p/\Z_p)^r$ 
for some $0\le r\le \mrm{genus}(C_{\finf})$.
\par 
Next, we study $\Hfl (C_{\finf},\Q_p /\Z_p )^{\lor}$.
Let us recall the following exact sequence$:$ 
\begin{equation}\label{equation:exactas}
H^i_{\mathrm{fl}} (C_{\finf},\Z /(p) \Z )\cong \mathrm{Ker}\left[ 
H^i_{\mathrm{fl}} (C_{\finf},\mc{O}_{C_{\finf}})
\xrightarrow{x\to x -x^p} 
H^i_{\mathrm{fl}} (C_{\finf},\mc{O}_{C_{\finf}})
\right] ,
\end{equation}
where the last map is induced by the Artin-Scherier sequence$:$ 
$$
0 \lra \Z /(p) \Z \lra \mathcal{O}_{C_{\finf}} 
\xrightarrow{x\to x -x^p}  
\mathcal{O}_{C_{\finf}} \lra 0 ,
$$
on the flat site over $C_{\finf}$. Since 
$H^{i-1}_{\mathrm{fl}} (C_{\finf},
\mathcal{O}_{C_{\finf}})$ 
is isomorphic to 
$H^{i-1}_{\mathrm{Zar}} (C_{\finf},
\mathcal{O}_{C_{\finf}})$, 
it is a finite dimensional $\oline{\mbb{F}}_q$-vector 
space. Hence $H^i_{\mathrm{fl}} (C_{\finf},
\mathcal{O}_{C_{\finf}} )
\xrightarrow{x\to x -x^p}
H^i_{\mrm{fl}} (C_{\finf},\mathcal{O}_{C_{\finf}} )$ 
is surjective for any $i$. 
Further, $H^1_{\mathrm{Zar}} (C_{\finf},\mathcal{O}_{C_{\finf}})$ 
is isomorphic to an $r$-dimensional vector space over $\oline{\mbb{F}}_q$. 
We have $H^1_{\mathrm{fl}} (C_{\finf},\Z /(p) \Z ) \cong 
(\Z /(p)\Z)^{\oplus r}$ by $(\ref{equation:exactas})$. On the other hand, 
$H^2_{\mrm{fl}} (C_{\finf},\Z /(p)\Z )$ is trivial 
since $H^2_{\mrm{fl}} (C_{\finf},\mathcal{O}_{C_{\finf}})=
H^2_{\mrm{Zar}} (C_{\finf},\mathcal{O}_{C_{\finf}})$ is trivial. 
This gives the following exact sequence: 
$$
0\lra 
H^1_{\mrm{fl}} (C_{\finf},\Z /(p) \Z ) \lra 
H^1_{\mrm{fl}} (C_{\finf},\Q_p /\Z_p ) \overset{\times p}{\lra} 
H^1_{\mrm{fl}} (C_{\finf},\Q_p /\Z_p ) \lra 0
$$
Since $H^1_{\mrm{fl}} (C_{\finf},\Q_p /\Z_p )^{\vee}/(p)
H^1_{\mrm{fl}} (C_{\finf},\Q_p /\Z_p )^{\vee}$ is finite, 
$H^1_{\mrm{fl}} (C_{\finf},\Q_p /\Z_p )^{\vee}$ is a finitely generated 
$\Z_p$-module by Nakayama's lemma. 
Further, $H^1_{\mrm{fl}} (C_{\finf},\Q_p /\Z_p )^{\vee}$ has no 
non-trivial $p$-torsion elements since 
$H^1_{\mrm{fl}} (C_{\finf},\Q_p /\Z_p )$ is $p$-divisible. 
Thus, $H^1_{\mrm{fl}} (C_{\finf},\Q_p /\Z_p )^{\vee}$ is a free 
$\Z_p$-module of rank $r$. 
This completes the proof of the claim and hence the proof 
of Lemma $\ref{lem:Finf}$ is completed. 
\end{pf}
By a remark given before Lemma $\ref{lem:Finf}$, the proof of 
Theorem $\ref{thm:cycmu}$ has been done in the ordinary case. 
\par 
Next, we discuss the supersingular case. 
By Lemma $\ref{pro:mainiso}$, 
we may assume the case where the following condition is satisfied: 
\vspace*{5pt}
\\ 
{\bf (SF)} $A/F$ is isomorphic to a supersingular abelian variety 
defined over a finite field $\mathbb{F}_q \subset F$ and 
the proper smooth curve $C_F$ which is the model of $F$ has 
invertible Hasse-Witt matrix. 
\vspace*{5pt}
\\ 
The assertion that $\mu (X(A/\finf))=0$ is equivalent to 
the assertion that $H^1_{\mrm{fl}} (C_{\finf} ,A\{ p\})$ is $p$-divisible. 
Now, we consider the following exact sequence: 
$$
H^1_{\mrm{fl}} (C_{\finf} ,A\{ p\})\overset{\times p}{\lra} 
H^1_{\mrm{fl}} (C_{\finf} ,A\{ p\}) \lra H^2_{\mrm{fl}} (C_{\finf},A[p]).
$$
Let $\bd{\alpha}_p$ be a finite group scheme defined to be 
the kernel of the frobenius map $F:\ \mbb{G}_a \lra \mbb{G}_a$. 
Since $A$ is supersingular, $A[p]$ is isomorphic to 
$\bd{\alpha}_p^{\oplus 2g}$, it suffices to show that 
$H^2 (C_{\finf},\bd{\alpha}_p)=0$. 
Now, from the exact sequence 
$$
H^1_{\mrm{fl}} (C_{\finf},\mc{O}_{C_{\finf}}) \overset{F}{\lra} 
H^1_{\mrm{fl}} (C_{\finf},\mc{O}_{C_{\finf}}) \lra 
H^2_{\mrm{fl}} (C_{\finf},\bd{\alpha}_p) \lra 0, 
$$
the assertion that $H^2 (C_{\finf},\bd{\alpha}_p)=0$ 
is equivalent to the assertion that the map $F$ on 
$H^1_{\mrm{fl}} (C_{\finf},\mc{O}_{C_{\finf}})$ is surjective. 
Since the last assertion is equivalent to the assumption that 
the Hasse-Witt matrix for $C_F$ is invertible, we complete the proof for 
the first assertion of Theorem $\ref{thm:cycmu}$. 
\par 
If the Hasse-Witt matrix for $C_F$ is not invertible, 
the kernel of $F$ on 
$H^1_{\mrm{fl}} (C_{\finf},\mc{O}_{C_{\finf}})$ contains 
a non-trivial $\oline{\mbb{F}}_p$-vector space and 
the Frobenius on $H^0_{\mrm{fl}} (C_{\finf},\mc{O}_{C_{\finf}})
\cong \oline{\mbb{F}}_p$ is surjective. 
Thus $H^1_{\mrm{fl}}(C_{\finf},A [p])^{\vee} \cong X(A/\fpinf)/(p)
X(A/\fpinf)$ has infinite $p$-rank and $\mu (X(A/\fpinf))$ has to be 
positive if the Hasse-Witt matrix for $C_F$ is not invertible. 
This completes the proof of Theorem $\ref{thm:cycmu}$. 
\section{A result over a $p$-adic Lie extension $L$}
\label{section:isotriv}
In this section, we prove Theorem $\ref{thm:isotriv}$. 
For the proof of Theorem $\ref{thm:isotriv}$, we introduce 
following new Selmer group: 
\begin{defn}
Let $A$ be an abelian variety over $F$ and let $S$ be a finite set of primes 
of $F$. By using the extension 
$0\lra A\{ p\}^{\mrm{conn}} \lra A\{ p\} \lra A\{ p\}^{\text{\'{e}t}} \lra 
0$ of p-divisible groups over $F$, we define the Selmer group over an 
algebraic extension $K$ over $F$ as follows:  
\begin{equation}\label{def:selgr}
\sel^{S}(A/K ) :=\mrm{Ker}\left[ H^1_{\mrm{fl}}(K ,A\{ p\})  
\lra \underset{v\in S_K}{\prod} 
H^1_{\mrm{fl}}(K_v ,A\{ p\}^{\text{\'{e}t}}) \times 
\underset{v\not\in S_K}{\prod} 
H^1_{\mrm{fl}}(K_v ,A) \right] ,
\end{equation}
where $S_K$ is a set of primes of $K$ lying over the primes in $S$. 
\end{defn}
From now on, 
we fix a $g$-dimensional abelian variety $A$ over $F$, which has 
good reduction outside a finite set $S$ of primes in $F$ 
and has ordinary reduction at each prime in $S$ as in 
Theorem $\ref{thm:isotriv}$. 
We give the following lemma on comparison between Selmer groups: 
\begin{lem}\label{lem:selcomp} 
Let the assumption be as in Theorem $\ref{thm:isotriv}$.
\begin{enumerate}
\item 
For any extension $K$ of $F$, $\sel (A/K)$ is naturally 
identified as a subgroup of $\sel^{S}(A/K)$. 
\item 
For any $K$ contained in $F_{\infty}$, 
the quotient $\sel^{S}(A/K )/\sel (A/K )$ is a cofinitely generated  
$\Z_p$-module whose corank is less than or equal to 
$g \sharp S$. 
\end{enumerate}
\end{lem}
\begin{proof}
For an extension $K$ of $F$, 
we denote by $B_{K_v}$ the subgroup of $H^1_{\mrm{fl}}(K_v,A\{ p\} )$ 
given by 
\small 
$$
B_{K_v}:=\mrm{Ker}\left[ H^1_{\mrm{fl}}(K_v,A\{ p\} ) \lra 
H^1_{\mrm{fl}}(K_v ,A\{ p\}^{\text{\'{e}t}}) \right] 
=\mrm{Im} \left[ H^1_{\mrm{fl}}(K_v,A\{ p\}^{\mrm{conn}} ) \lra 
H^1_{\mrm{fl}}(K_v ,A\{ p\}) \right] .
$$
\normalsize 
Recall that $A(K_v )\otimes \Q_p /\Z_p $ is regarded as a subgroup of 
$H^1_{\mrm{fl}}(K_v,A\{ p\} )$ and the kernel of 
$H^1_{\mrm{fl}}(K_v,A\{ p\} )\lra H^1_{\mrm{fl}}(K_v,A )$ 
is equal to $A(K_v )\otimes \Q_p /\Z_p $. 
By comparing the definition of $\sel (A/K)$ and $\sel^{S} (A/K)$ 
($(\ref{equation:seldef})$ and $(\ref{def:selgr})$, respectively), 
$\sel (A/K)$ is a subgroup of $\sel^{S} (A/K)$ if and only if 
$A(K_v )\otimes \Q_p /\Z_p $ is contained in $B_{K_v}$ for every $v\in S$. 
We denote by $\mc{A}_v$ the Neron model on the ring of integers of $K_v$ 
associated to the abelian variety $A/K_v$.
Let $\widehat{\mc{A}}_v$ be the formal completion 
of $\mc{A}_v$ along the zero section of the special fiber 
of $\mc{A}_v$. Since $A$ has ordinary reduction at $v$, 
$\widehat{\mc{A}}_v[p^n]$ is a finite flat group scheme 
which is \'{e}tale locally isomorphic to 
a finite flat group scheme $(\mu_{p^n})^{g \sharp S}$ over $\mc{O}_v$ 
whose generic fiber is isomorphic to the connected part $A[p^n]^{\mrm{conn}}$ 
of the finite flat group scheme $A[p^n]$ over $K_v$. 
We have the following diagram: 
\begin{equation}\label{equation:flgp}
\begin{CD}
0 @>>> \widehat{\mc{A}}_v ( \mc{O}_v )\otimes_{\Z}\Q_p /\Z_p  
@>{\delta}>> \varinjlim_n 
H^1_{\mrm{fl}} (K_v , \widehat{\mc{A}}_v [p^n]) @>{\epsilon}>> 
H^1_{\mrm{fl}} (K_{v} ,\widehat{\mc{A}}_v)\\ 
@. @V{\alpha }VV @VV{\beta}V @VV{\gamma}V \\
0 @>>> A(K_{v})\otimes_{\Z}\Q_p /\Z_p  @>>> 
H^1_{\mrm{fl}} (K_{v} ,A\{ p\}) @>>> 
H^1_{\mrm{fl}} (K_{v} ,A). \\ 
\end{CD}
\end{equation}
Since $\mc{A}_v (\mc{O}_v) =A(K_v )$, we have an exact sequence: 
$$
0 \lra \widehat{\mc{A}}_v ( \mc{O}_v ) \lra A(K_{v}) \lra \oline{\mc{A}}_v 
(\mbb{F}_v), 
$$
where $\mbb{F}_v$ is the residue field of $K_v$ and 
$\oline{\mc{A}}_v$ is $\mc{A}_{v} \times 
_{\mathrm{Spec}(\mc{O}_v)} \mathrm{Spec}(\mbb{F}_v)$. 
Since $\oline{\mc{A}}_v (\mbb{F}_v)$ 
is a direct limit of finite groups, we have 
$\oline{\mc{A}}_v (\mbb{F}_v) \otimes \Q_p /\Z_p =0$. 
This implies the following assertion:   
\begin{equation}\label{equation:assertion}
\text{The map $\alpha$ in the diagram 
$(\ref{equation:flgp})$ must be surjective.}
\end{equation} Consequently, 
$A(K_v )\otimes \Q_p /\Z_p $ is a subgroup of $B_{K_v}$ 
for every $v\in S$. 
This completes the proof of the first assertion of the lemma. 
\par 
Next, we prove the second assertion. 
Suppose that $K$ is contained in $F_{\infty}$. 
Consider the following diagram: 
\small 
\begin{equation}\label{equation:twocomp}
\begin{CD}
0 @>>> \sel (A /K )  @>>> 
H^1_{\mrm{fl}} (K , A\{ p\} ) @>>> 
\underset{v}{\prod} 
\dfrac{H^1_{\mrm{fl}}(K_v ,A\{ p\})}
{A(K_v) \otimes \Q_p /\Z_p } \\ 
@. @VVV @VVV @VVV \\
0 @>>> \sel^S  (A /K )   @>>> 
H^1_{\mrm{fl}} (K ,A\{ p\}) @>>> \hspace*{-2pt}
\underset{v\in S}{\prod} 
\dfrac{H^1_{\mrm{fl}}(K_v ,A\{ p\})}
{B_{K_v}} \times 
\underset{v\not\in S }{\prod} 
\dfrac{H^1_{\mrm{fl}}(K_v ,A\{ p\})}{A(K_v) 
\otimes \Q_p /\Z_p }, \\ 
\end{CD}
\end{equation}
\normalsize
Thus, $\sel^S (A/K)/\sel (A/K )$ is a subquotient of 
$\underset{v\in S}{\prod} 
\dfrac{B_{K_v}}{A(K_v)\otimes \Q_p /\Z_p}$
by the snake lemma. 
Since $\dfrac{B_{K_v}}{A(K_v)
\otimes \Q_p /\Z_p}$ 
is a quotient of 
\begin{equation}\label{equation:cdef}
C_{K_v}:=\mrm{Coker}\left[ 
\widehat{\mc{A}}_v ( \mc{O}_v )\otimes_{\Z}\Q_p /\Z_p 
\overset{\delta}{\lra} \varinjlim_n 
H^1_{\mrm{fl}}(K_v , \widehat{\mc{A}}_v [p^n])\right] ,
\end{equation}
it suffices to show that 
$\underset{v\in S}{\prod} C_{K_v}$ 
is cofinitely generated of corank less than or equal to 
$g\sharp S$ over $\Z_p$. 
Let $K^{\mrm{ur}}_v$ be the maximal unramified extension 
of $K_v$. 
By the Hochshild-Serre spectral sequence, 
we have the following exact sequence.
\begin{multline}\label{equation:snakeord}
0 \lra 
\varinjlim_n H^1 (\Gamma ,\widehat{\mc{A}}_v  [p^n](\mc{O}^{\mrm{ur}}_v))
\lra 
\varinjlim_n H^1_{\mrm{fl}} (K_v , \widehat{\mc{A}}_v  [p^n])
\lra 
H^1_{\mrm{fl}} (K^{\mrm{ur}}_v , \widehat{\mc{A}}_v  [p^n])
^{\Gamma} \\ 
\lra 
\varinjlim_n H^2 (\Gamma ,\widehat{\mc{A}}_v  [p^n](\mc{O}^{\mrm{ur}}_v))
\lra 0 
\end{multline} 
where $\Gamma =\mrm{Gal}(K^{\mrm{ur}}_v /K_v)$. 
By the assumption that $A$ is ordinary at $v\in S$, 
$\widehat{\mc{A}}_v (\mc{O}^{\mrm{ur}}_v)$ is isomorphic to 
$(U^1_{K^{\mrm{ur}}_v})^g$, where $U^1_{K^{\mrm{ur}}_v}
\subset \mc{O}_{K^{\mrm{ur}}_v}^{\times}$ is the group of 
principal units. Since $\widehat{\mc{A}}_v [p^n](\mc{O}^{\mrm{ur}}_v)=0$, 
$(\ref{equation:snakeord})$ implies the following isomorphism: 
\begin{equation}\label{equation:urisom}
\varinjlim_n H^1_{\mrm{fl}} (K_v , \widehat{\mc{A}}_v  [p^n]) \cong 
\varinjlim_n H^1_{\mrm{fl}} (K^{\mrm{ur}}_v , \widehat{\mc{A}}_v  [p^n])
^{\Gamma} .
\end{equation}
Since 
$(\widehat{\mc{A}}_v (\mc{O}^{\mrm{ur}}_v))^{\Gamma} =
\widehat{\mc{A}}_v (\mc{O}_v )$, we 
have the following commutative diagram: 
$$ 
\begin{CD}
0 @>>> \widehat{\mc{A}}_v ( \mc{O}_v ) @>{\times p}>> 
\widehat{\mc{A}}_v ( \mc{O}_v ) @>>> 
\widehat{\mc{A}}_v ( \mc{O}_v )\otimes_{\Z} \Z /p^n \Z @>>> 0\\ 
@. @| @| @VVV @. \\
0 @>>> (\widehat{\mc{A}}_v (\mc{O}^{\mrm{ur}}_v))
^{\Gamma} @>{\times p}>> 
(\widehat{\mc{A}}_v (\mc{O}^{\mrm{ur}}_v))
^{\Gamma} @>>> 
(\widehat{\mc{A}}_v (\mc{O}^{\mrm{ur}}_v)\otimes_{\Z} \Z /p^n \Z)
^{\Gamma} @>>> 0. 
\end{CD}
$$ 
The surjectivity of the last map in the lower line of the diagram 
follows from the fact that 
$H^1 (\Gamma , \widehat{\mc{A}}_v (\mc{O}^{\mrm{ur}}_v))$ is trivial. 
Hence, we have $
\widehat{\mc{A}}_v ( \mc{O}_v )\otimes_{\Z} \Z /p^n \Z  \cong 
(\widehat{\mc{A}}_v (\mc{O}^{\mrm{ur}}_v)\otimes_{\Z} \Z /p^n \Z ) 
^{\Gamma}$. Thus, by taking the inductive limit with respect to $n$, we have:
\begin{equation}\label{equation:aurisom}
\widehat{\mc{A}}_v ( \mc{O}_v )\otimes_{\Z} \Q_p /\Z_p  \cong 
(\widehat{\mc{A}}_v (\mc{O}^{\mrm{ur}}_v)\otimes_{\Z} \Q_p /\Z_p ) 
^{\Gamma}
\end{equation}
We have 
\begin{equation}\label{equation:aurcoker}
\begin{split}
& \mathrm{Coker}
\left[ 
\widehat{\mc{A}}_v ( \mc{O}^{\mathrm{ur}}_v )\otimes_{\Z}\Q_p /\Z_p 
\overset{\delta}{\lra} \varinjlim_n 
H^1_{\mrm{fl}}(K^{\mathrm{ur}}_v , \widehat{\mc{A}}_v [p^n])\right] \\ 
& \cong \mathrm{Coker} 
\left[ (U^1_{K^{\mrm{ur}}_v} \otimes_{\Z} \Q_p /\Z_p )^g \lra 
 H^1_{\mrm{fl}}(K^{\mrm{ur}}_v , (\mu_{p^n})^g ) \right] \\
& \cong \mathrm{Coker}
\left[ (U^1_{K^{\mrm{ur}}_v} \otimes_{\Z} \Q_p /\Z_p )^g \lra 
((K^{\mrm{ur}}_v)^{\times} \otimes_{\Z} \Q_p /\Z_p )^g \right] \cong 
(\Q_p /\Z_p )^g
\end{split}
\end{equation}
By $(\ref{equation:cdef})$, $(\ref{equation:urisom})$, 
$(\ref{equation:aurisom})$ and $(\ref{equation:aurcoker})$, 
$C_{K_v}$ is a subgroup of $(\Q_p /\Z_p )^g$ for each $v\in S$.
This completes the proof of the lemma. 
\end{proof} 
We have the following lemma: 
\begin{lem}\label{lem:Hcoinv}
Under the same assumption as that of Theorem $\ref{thm:isotriv}$, 
the group $X^{S} (A/L )_{H} \cong (\sel^{S} 
(A/L )^{H})^{\vee}$ 
is a finitely generated $\Z_p$-module.
\end{lem}
\begin{proof}
Let $S$ be a finite set of primes in $\fpinf$ where the extension 
$L/\fpinf$ is ramified. 
Since $A$ has good reduction outside $S$ and since $L/\fpinf$ 
is unramified outside $S$, 
we have the following commutative diagram: 
\small 
$$ 
\begin{CD}
0 @>>> \sel^{S} (A/\fpinf) @>>> 
H^1_{\mrm{fl}} (C_{\fpinf}\setminus S ,A\{ p\}) @>>> 
\underset{v \in S}{\prod} 
H^1_{\mrm{fl}} (F^{(p)}_{\infty ,v} ,A\{ p\}^{\text{\'{e}t}})\\ 
@. @V{a}VV @V{b}VV @VV{c}V \\
0 @>>> \sel^{S} (A/L )^{H} @>>> 
\left( \uset{K}{\varinjlim} 
H^1_{\mrm{fl}} (C_{K}\setminus S_K ,A \{ p \}) \right)^{H} @>>> 
\left( 
\underset{w\in S_L}{\prod}
H^1_{\mrm{fl}} (L_{w} ,A\{ p\}^{\text{\'{e}t}})
\right)^{H}, \\ 
\end{CD} 
$$ 
\normalsize 
where $K$ runs a finite extension of $\fpinf $ contained in 
$L$ and $S_K$ is the set of primes of $K$ lying over $S$.
By the Hochshild-Serre spectral sequence, 
we have the following exact exact sequence: 
\begin{multline*}
H^1 (H,A \{ p \}(L)) \lra 
H^1_{\mrm{fl}} (C_{\fpinf} \setminus S ,A\{ p\}) \\
 \overset{b}{\lra} 
\left( \uset{\fpinf \subset K \subset L}{\varinjlim} 
H^1_{\mrm{fl}} (C_{K} \setminus S_K ,A \{ p \}) \right)^{H} \lra 
H^2 (H,A \{ p \}(L)). 
\end{multline*}
Since $H$ is a $p$-adic Lie group and $A \{ p \}(L)$ is an abelian group 
which is cofinite type over $\Z_p$, 
$H^i (H,A \{ p \}(L) ) ^{\vee}$ 
is finitely generated over $\Z_p$ for every $i$. 
In the same way, the Pontrjagin dual of 
$\mrm{Ker}(c) \cong \underset{v\in S_{\fpinf }}{\prod} 
\underset{w \vert v}{\prod} 
H^1 (H_w , A\{ p\}^{\text{\'{e}t}}(L_w ))$ is a finitely generated 
$\Z_p$-module, where $H_w \subset H$ is the decomposition subgroup 
at $w$ of $L$. 
The module $X(A/\fpinf )= \sel (A/ \fpinf )^{\vee}$ is 
finitely generated over $\Z_p$ by the assumption of 
Theorem $\ref{thm:isotriv}$. 
By Lemma $\ref{lem:selcomp}$, $X^{S}(A/\fpinf )= 
\sel^{S} (A/ \fpinf )^{\vee}$ is also a 
finitely generated $\Z_p$-module. 
This completes the proof of Lemma. 
\end{proof}
Now, we will complete the proof of 
Theorem $\ref{thm:isotriv}$ by using the following result of 
Balister and Howson: 
\begin{lem}\cite[Corollary in \S 3]{BH}\label{lem:bh} 
Let $H$ be a $p$-adic Lie group and $X$ a compact $\La (H)$-module. 
Let $X_{H}$ denote the $H$-coinvariant quotient of 
$X$. If $X_{H}$ is a $\Z_p$-module of finite type, 
then $X$ is a $\La (H)$-module of finite type.
\end{lem} 
In fact, since $X(A/L)_H$ is a quotient of $X^{S} (A/L )_{H}$, 
$X(A/L)_H$ is a finitely generated $\Z_p$-module. 
Hence $X(A/L)$ is a $\La (H)$-module of finite type by using the above lemma.

\end{document}